\documentclass[10 pt]{amsart}

\usepackage{amsmath}
\usepackage{amsfonts}
\usepackage{amssymb}
\usepackage{amstext}
\usepackage{amsbsy}
\usepackage{amsopn}
\usepackage{amsthm}
\usepackage{amsxtra}
\usepackage{graphicx}
\usepackage{hyperref}
\usepackage{color}
\usepackage[all]{xy}
\usepackage{enumitem}
\usepackage{tikz-cd}
\usepackage{mathtools}
\usepackage{array}
\usepackage{comment}
\usepackage{bm}
\usepackage{soul}

\newtheorem{theorem}{Theorem}[section]
\newtheorem{lemma}[theorem]{Lemma}
\newtheorem{proposition}[theorem]{Proposition}
\newtheorem{corollary}[theorem]{Corollary}
\newtheorem{question}[theorem]{Question}

\theoremstyle{definition}
\newtheorem*{definition*}{Definition}
\newtheorem{definition}[theorem]{Definition}

\newtheorem{example}[theorem]{Example}

\theoremstyle{remark}
\newtheorem{remark}[theorem]{Remark}
\newtheorem*{remark*}{Remark}

\newcommand{\R}{\mathbb{R}}
\newcommand{\Z}{\mathbb{Z}}

\newcommand{\N}{\mathbb{N}}

\newcommand{\CA}{\mathcal{A}}
\newcommand{\CB}{\mathcal{B}}

\newcommand{\CI}{\mathcal{I}}
\newcommand{\CL}{\mathcal{L}}
\newcommand{\CM}{\mathcal{M}}

\newcommand{\CS}{\mathcal{S}}

\newcommand{\CK}{\mathcal{K}}
\newcommand{\CJ}{\mathcal{J}}
\newcommand{\CF}{\mathcal{F}}
\newcommand{\CC}{\mathcal{C}}

\newcommand{\CE}{\mathcal{E}}

\newcommand{\one}{\boldsymbol{1}}

\newcommand{\aut}{\textnormal{Aut}}

\newcommand{\Fix}{{\rm{Fix}}}

\newcommand{\supp}{\textnormal{supp}}
\newcommand{\sym}{\textnormal{Sym}}

\newcommand{\fin}{\textnormal{fin}}
\newcommand{\ct}{\textnormal{ct}}
\newcommand{\diam}{\textnormal{diam}}

\newcommand{\autinf}{\textnormal{Aut}^{(\infty)}}

\newcommand{\inv}[1]{\textnormal{Inv}(\sigma_{n})}
\newcommand{\infinv}[1]{\textnormal{Inv}^{\infty}(\sigma_{n})}

\newcommand{\Id}{{\rm Id}}

\newcommand{\Homeo}{{\rm Homeo}}

\newcommand{\T}{{\mathbb T}}

\newcommand{\Ext}{\textrm{Ext}}

\usepackage{todonotes}
\begin{document}

\title{Invariant random compacts}
\author{Bryna Kra}
\address{Department of Mathematics, Northwestern University, 2033 Sheridan Road, 
Evanston, IL 60208}
\email{kra@math.northwestern.edu}

\author{Scott Schmieding}
\address{Department of Mathematics, 
107 McAllister Building
University Park, State College, PA 16802}
\email{sks7247@psu.edu}

\thanks{BK was partially supported by the Simons Foundation and NSF grant DMS-2348315 
and SS was partially supported by NSF grant DMS-2247553.}

\begin{abstract}
For a compact metric space $X$ with a group $G$ acting on it continuously, an invariant random compact is a Borel probability measure on the space of nonempty compact subsets of $X$ that is invariant under the action of $G$. The action is IC-rigid if, with respect to every invariant random compact, every compact set is almost surely either finite or $X$. 
We give sufficient conditions for an action to be IC-rigid, and show there are natural examples of such actions. We further consider a notion of weak IC-rigidity, and prove that the Chacon system is weakly IC-rigid but not IC-rigid. As an application, we prove results concerning multiplicative largeness of dilations of sets on the circle. 
\end{abstract}
\maketitle

\section{Introduction} 
\subsection{Summary of the main results}

Consider a group, or more generally a semigroup, $G$ acting continuously on a compact metric space $X$. 
The study of the induced action of $G$ on $\CK(X)$, the space of nonempty compact subsets of $X$ endowed with the Hausdorff metric, has a long history (see for example~\cite{BauerSigmund, Banks, glasner, GW1, LOYZ}). 
This induced action captures topological information about the orbits of compact sets in $\CK(X)$, and it is natural to consider the statistical distribution of the orbits of compact sets. 
To do so, we study invariant measures for the induced action on $\CK(X)$ and their relation to the dynamics of the $G$-action on $X$. 

We call a Borel probability measure $\mu$ on $\CK(X)$ which is invariant under the action of $G$ on $\CK(X)$ an \emph{invariant random compact} (IRC for short). This terminology is partly inspired by analogy with the notion of an invariant random subgroup of a group (see for example~\cite{AGV, ABB+}), which is a conjugation-invariant measure on the space of subgroups of a group.  
A compact subset $Y\subset X$ that is invariant under $G$, meaning that $gY = Y$ for all $g \in G$,  gives rise to an IRC by considering the delta mass $\delta_Y$. In this sense, we view an IRC as a generalization of an invariant subset for the action.

Any system with a $G$-invariant Borel probability measure has an IRC in a trivial way: pushing such a measure $\mu$ forward via the map $X \to \CK(X)$ defined by $x \mapsto \{x\}$ yields an IRC for the system. Such an IRC is supported on the set of singleton subsets in $\CK(X)$.  
Analogously, there are IRCs supported on the subset $\CK_{\fin}(X)\subset\CK(X)$ of finite subsets of $X$ arising from invariant measures for the diagonal action on the product of finitely many copies of $X$ with itself. We refer to these IRCs, meaning ones supported on the collection of finite subsets, as \emph{finitary}. The finitary IRCs are connected to self-joinings of the system, and in the amenable case, all finitary IRCs are combinations of pushforwards of invariant measures for the diagonal actions on self-products of $X$ (see Section~\ref{sec:finitaryircs} for these results). 

Thus we focus on IRCs that do not arise in this manner:  IRCs $\mu$ for which $\mu$-almost every compact subset of $X$ is infinite. We call such IRCs \emph{nonfinitary}. Assuming $X$ is infinite and each $g \in G$ acts surjectively on $X$ (we make this mild assumption, which always holds when $G$ is a group, throughout), there is always a nonfinitary IRC, namely $\delta_{X}$. 
More generally, an infinite compact invariant set $Y\subset X$ yields a nonfinitary IRC $\delta_{Y}$. However, even among minimal systems, typically, there are nonfinitary IRCs that are not $\delta_{X}$ (see Section~\ref{sec:examples-of-IRC} for several examples).  For instance, any amenable group acting by isometries, or any system having an invariant measure for which the action is essentially free and has a rigidity sequence, has nonfinitary IRCs other than $\delta_{X}$ (see Proposition~\ref{prop:act-by-isometry} and Theorem~\ref{thm:rigiditygivesIRCs}).  

Motivated by this, we say that the action of $G$ on $X$ is \emph{IC-rigid} if the only nonfinitary IRC is $\delta_{X}$. 
It is immediate that an infinite IC-rigid system must be \emph{almost minimal}, meaning every proper compact $G$-invariant set is finite. 
But IC-rigidity is a stronger notion, and this raises the question of whether there are natural examples of such systems.

To answer this, we introduce two strong forms of transitivity for actions on Cantor spaces, called \emph{deep transitivity} and \emph{extreme transitivity} (see Section~\ref{sec:deeply-extremely} for the precise definitions). 
In Section~\ref{sec:finitaryircs}, we prove that deeply transitive actions are IC-rigid, and completely classify all of their IRCs.
To state the result, for an action of $G$ on $X$, let $\CM_{G}(X)$ denote the space of $G$-invariant Borel probability measures on $X$.
\begin{theorem}
\label{th:deeply}
If the group $G$ acts deeply transitively on a Cantor set $X$, then all of the following hold:
\begin{enumerate}
\item
The action is IC-rigid.
\item
If $\CM_{G}(X) = \emptyset$, then the induced action of $G$ on $\CK(X)$ is uniquely ergodic with unique invariant measure $\delta_{X}$.
\item
\label{item:the-measures}
If $\CM_{G}(X) \ne \emptyset$, then the simplex $\CM_{G}(\CK(X))$ is isomorphic to the infinite-dimensional probability simplex ${\Delta(\mathbb{N}) =  \{(x_{i})_{i=0}^{\infty} : 0 \le x_{i} \textrm{ and } \sum_{i=0}^{\infty}x_{i}=1\}}$. 
\end{enumerate}
\end{theorem}
Furthermore, the description of the ergodic measures in part~\eqref{item:the-measures} of this theorem is  explicit; see Theorem~\ref{cor:deeply}.

Returning to the question of familiar systems which are IC-rigid, in Section~\ref{sec:ICridigexamples} we give several examples of well-studied group actions which are deeply transitive. This includes the action of stabilized automorphism groups of full shifts, the action of Higman-Thompson groups on one-sided full shifts, as well as certain AF full groups acting on Cantor spaces. We note that in each of these examples, there exists an invariant Borel probability measure for the action.

Extremely transitive actions are deeply transitive. For extremely transitive actions, we also give a complete classification of all of the self-joinings. In fact, all of the examples we give in Section~\ref{sec:ICridigexamples} are also 
extremely transitive, yielding a classification of their self-joinings (see Section~\ref{sec:selfjoiningsexttransitivity}). The AF full groups also act by prefix-permutation on real numbers, and we use the extreme transitivity together with Theorem~\ref{th:deeply} to prove a result about $\varepsilon$-density of images of infinite sets under random permutations (Theorem~\ref{thm:prefixpermutcircle}). 

The definition of IC-rigidity can be weakened: an action of $G$ on $X$ is \emph{weakly IC-rigid} if $\mu$-almost every compact subset of $X$ is either countable or $X$ for every IRC $\mu$ of the system. This definition requires some care, as a theorem of Hurewicz shows that for $X$ uncountable, the set of uncountable subsets of $X$ is $\Sigma_{1}^{1}$-complete in $\CK(X)$ and hence not Borel, but we defer the technical discussion of these issues until Section~\ref{sec:weakly-rigid}. The difference between weakly IC-rigid and IC-rigid is nontrivial, and is witnessed by the Chacon system.

\begin{theorem}\label{thm:chaconintro}
The Chacon system is weakly IC-rigid but not IC-rigid.
\end{theorem}
Thus the Chacon system possesses nonfinitary IRCs besides $\delta_{X}$, and with respect to such IRCs, almost surely every compact set is either countable or $X$. 

The weak IC-rigidity of Chacon is a consequence of a more general result we prove in Theorem~\ref{thm:msjweaklyicrigid}, which is of interest in its own right: a minimal, uniquely ergodic system satisfying a certain property which we call countable exceptions for products is weakly IC-rigid. That Chacon satisfies these conditions was proved in~\cite{DRS}.
On the other hand, proving that the Chacon  system is not IC-rigid relies on more subtle combinatorial properties of its structure as a rank one system, 
which we carry out in Section~\ref{sec:Chacon-not-rigid}.

We use the notion of weak IC-rigidity to derive results about dilations on the circle. Consider the multiplicative semigroup $\mathbb{N}$ acting on the circle $\mathbb{T} = \mathbb{R}/\mathbb{Z}$ by multiplication.  Glasner~\cite{glasner} showed that if $A\subset\T$ is infinite and $\varepsilon > 0$, then there is some integer $n\geq 1$ such that the dilation $\{na: a\in A\}$ is $\varepsilon$-dense in $\T$.  This property has been studied in a variety of settings (see for example~\cite{Ber-Bos, Bu-Fi, KL}).  Stronger quantitative versions are given in~\cite{AP, Ber-Per}, and Berend and Peres~\cite{Ber-Per} show that for every infinite compact set $A \subset \mathbb{T}$ there exists a sequence $n_{i}$ in $\mathbb{N}$ of natural density one such that $n_{i}A \to \mathbb{T}$ in $\CK(\mathbb{T})$. Rephrasing their result, the sets $F_{m} = [1,m]$ form a F\o{}lner sequence for the additive semigroup $\mathbb{N}$, and Berend and Peres show the existence of a set $I \subset \mathbb{N}$ such that $$\lim_{m\to\infty}\frac{|I \cap F_{m}|}{|F_{m}|}\to 1  \quad \text{ and } \quad \lim_{n \to \infty, n \in I} nY \to \mathbb{T} \text{ in } \CK(\mathbb{T}).$$
Viewed in this way, their result concerns the additive largeness of a set of dilations of $Y$ converging to $\mathbb{T}$. We prove a result for sets which are multiplicatively large. For $p \in \mathbb{N}$, let $T_{p} \colon \mathbb{T} \to \mathbb{T}$ denote the map $T_{p}(x) = px \mod 1$.

\begin{theorem}\label{thm:dilationsintro1}
Let $Y \subset \mathbb{T}$ be a compact subset invariant under $T_{p}$ for some $p \ge 2$, and suppose $T_{p} \colon Y \to Y$ has positive topological entropy. Then for every F\o{}lner sequence $F_{m}$ of the multiplicative semigroup $\mathbb{N}$, there exists a set $J \subset \mathbb{N}$ such that
$$\lim_{m\to\infty}\frac{|J \cap F_{m}|}{|F_{m}|} \to 1 \quad \text{ and } \quad \lim_{n \to \infty, n \in J}nY = \mathbb{T} \textrm{ in } \CK(\mathbb{T}).$$
\end{theorem}

The theorem does not hold if one considers arbitrary, even uncountable, compact subsets (see Example~\ref{example:dilationset}). Along similar lines we prove the following.

\begin{corollary}\label{cor:dilationsintro2}
Let $Y \subset \mathbb{T}$ be a compact subset invariant under $T_{p}$ for some $p \ge 2$, and suppose $T_{p} \colon Y \to Y$ has positive topological entropy. Then for every $\varepsilon > 0$, the set
$$\{n \in \mathbb{N} : nY \textrm{ is }\varepsilon\textrm{-dense in } \mathbb{T}\}$$
is a multiplicatively syndetic subset of $\mathbb{N}$.
\end{corollary}

Our methods for these results are completely different from those of Berend and Peres. We make use of IRCs, and the main step is to prove that the action of $\mathbb{N}$ on $\mathbb{T}$ is weakly IC-rigid. As a second step, we must work to avoid the orbit of $Y$ distributing in a way which assigns mass to countable subsets. This is where the positive entropy assumption is used.  

\subsection{Questions}
The study of invariant random compacts leads to several natural directions for further research.

The first asks whether the weak IC-rigidity for the multiplicative action of $\mathbb{N}$ on $\mathbb{T}$ can be upgraded to IC-rigidity.

\begin{question}
\label{conj:IRC-N}
Is $\delta_{\mathbb{T}}$ the only nonfinitary invariant random compact for the action of the multiplicative semigroup $\N$ on the torus $\T = \R/ \Z$?
\end{question}

A positive answer to Question~\ref{conj:IRC-N} would yield stronger versions of the results in Section~\ref{subsec:dilations}.

Furstenberg~\cite{furstenberg} showed that if $p,q\geq 2$ are multiplicatively independent natural numbers, then the only compact infinite subset of $\mathbb{T}$ invariant under multiplication by $p$ and $q$ is $\mathbb{T}$. In light of this, it is natural to consider whether these actions are IC-rigid.

\begin{question}  
\label{conj:IRC-non-lac}
If $p,q \ge 2$ are multiplicatively independent natural numbers, is $\delta_{\mathbb{T}}$ the only nonfinitary invariant random compact for the action of multiplication by $p$ and $q$ on the torus $\mathbb{T}$?
\end{question}
Clearly, a positive answer to Question~\ref{conj:IRC-N} implies a positive answer to Question~\ref{conj:IRC-non-lac}. More generally, there is the problem of determining all of the IRCs for the action of multiplication by $p$ and $q$ on the torus $\mathbb{T}$. The case of finitary IRCs supported on singleton sets is a trivial reformulation of Furstenberg's question about invariant measures for multiplication by $p$ and $q$ on the $\mathbb{T}$. Maucourant's result~\cite{maur} shows there are finitary IRCs for these systems, supported on sets of size at least four, which do not arise by pushing forward Lebesgue on an entire torus.

The automorphism group $\aut(\sigma_{n})$ of a full shift $(X_{n},\sigma_{n})$ is the group of self-homeomorphisms of $X_{n}$ which commute with the shift $\sigma_{n}$. The stabilized automorphism group $\autinf(\sigma_{n})$ of $(X_{n},\sigma_{n})$ is the group of self-homeomorphisms of $X_{n}$ which commute with $\sigma_{n}^{m}$ for some $m$. For $n \ge 2$, the group $\autinf(\sigma_{n})$ acts deeply transitively on $X_{n}$ (see Subsection~\ref{sec:ICridigexamples}), and so by Theorem~\ref{th:deeply} the action is IC-rigid. The group $\aut(\sigma_{n})$ does not act deeply transitively on $X_{n}$ however, and we ask the following.

\begin{question}
Is the action of the automorphism group $\aut(\sigma_{n})$ on the full shift $X_{n}$ IC-rigid?
\end{question}

The action of the Higman-Thompson groups on one-sided full shifts give examples of IC-rigid actions of finitely-generated groups. However, we do not know of any examples of IC-rigid actions where the acting group is more constrained: for example, abelian. 

\begin{question}
Does there exist an IC-rigid action of an abelian group? In particular, does there exist an IC-rigid $\mathbb{Z}$-action?
\end{question}

\subsection{Brief outline of the paper}
We start in Section~\ref{sec:background} with a summary of the dynamical background necessary to precisely state our results. 
 In Section~\ref{sec:ircs}, we define our main object of study, the invariant random compacts. We use the remainder of the section to develop the basic properties of IRCs, exhibit examples, and define IC-rigidity. In Theorem~\ref{thm:rigiditygivesIRCs} we show that measure theoretic rigidity gives rise to nontrivial nonfinitary IRCs. 
As a corollary, we note that this implies that a generic set of subshifts, considered in the space of infinite transitive subshifts of any full shift, have nontrivial nonfinitary IRCs. 

In Section~\ref{sec:weakening}, we introduce a weaker notion of IC-rigidity, and prove Theorem~\ref{thm:chaconintro}. In Section~\ref{sec:dilations} we turn to dilations of the circle. We prove weak IC-rigidity of the action of $\N$ on $\mathbb{T}$ by multiplication, and also classify the self-joinings of this system. We then prove several results about dilations including Theorem~\ref{thm:dilationsintro1} and Corollary~\ref{cor:dilationsintro2}.

Section~\ref{sec:treestructuresdeeptransitivity} introduces deep and extreme transitivity, which are defined via tree structures. Here examples of extremely (and hence deeply) transitive actions are given, and we then prove the IC-rigidity of deeply transitive actions.

Section~\ref{sec:finitaryircs} concerns the finitary IRCs for deeply transitive actions, where we complete the proof of Theorem~\ref{th:deeply}. Section~\ref{sec:selfjoiningsexttransitivity} is devoted to the classification of all self-joinings of extremely transitive actions.

\section{Background}\label{sec:background}
\subsection{Systems}
By a {\em topological dynamical system} $(X,G)$, or {\em system} for short, we mean a group or semigroup $G$ acting continuously on a compact metric space $X$. Throughout, we assume that all groups and semigroups are countable, the space $X$ is infinite, and let $d_X$ denote the metric on $X$. In the case of a semigroup, we always assume each $g \in G$ acts surjectively on $X$, and that all semigroups are left cancellative. The system is {\em minimal} if $\overline{\{gx: g\in G\}} = X$ for all $x\in X$ and is {\em transitive} if 
 $\overline{\{gx: g\in G\}} = X$ for some $x\in X$.
 The system $(X,G)$ is {\em totally minimal} if for every finite index subgroup $H \subset G$, the action of $H$ on $X$ is also minimal.
We say $Y \subset X$ is {\em invariant} if $g(Y) = Y$ for every $g \in G$. For a set $A \subset X$ and $\varepsilon > 0$ we denote the ball of radius $\varepsilon$ around $A$ by $B_{\varepsilon}(A) = \{x \in X : d(x,A) < \varepsilon\}$. For a $\mathbb{Z}$-action defined by a homeomorphism $T \colon X \to X$ we often write simply $(X,T)$ for the system.

Throughout, we assume that for a semigroup $G$, every $g\in G$ acts surjectively on $X$ and that 
 every $g \in G$ acts finite-to-one,
in the sense that $g^{-1}(x)$ is finite for every $x \in X$.  These assumptions obviously always hold when $G$ is a group, and we assume them for semigroup actions for technical reasons. In particular, they hold for the multiplicative action of the semigroup $\N$ on the circle $\T$.

For a compact metric space $X$, we let $\CM_G(X)$ denote  the set of all $G$-invariant Borel probability measures on $X$. 
By a {\em measure-preserving system}, we mean a triple $(X, \mu, G)$ where $(X, G)$ is a topological dynamical system and $\mu\in\CM_G(X)$.  Throughout, we assume that for any measure-preserving system, the space $X$ is a compact metric space and the system is endowed with the Borel $\sigma$-algebra (we omit the $\sigma$-algebra from the notation). 
The set $Y\subset X$ is {\em $G$-invariant} if $gY = Y$ for all $g\in G$, where as usual equality is meant up to sets of measure zero. 
A measure-preserving system is {\em ergodic} if any $G$-invariant Borel subset of $X$ has measure zero or one, and we let $\CM_G^e(X)$ denote the subset of ergodic measures in $\CM_G(X)$.

If $(X,G)$ and $(Y,G)$ are systems, then a continuous map $f \colon X \to Y$ is \emph{equivariant} with respect to the actions if $f \circ g = g \circ f$ for all $g \in G$. An equivariant map $f \colon X \to Y$ which is onto is called a \emph{factor map} and we say that $(Y, G)$ is a {\em factor} of $(X, G)$ (note that we make the usual slight abuses of notation, using the same letter to denote the action on $X$ and $Y$ and often omitting the composition symbol $\circ$ from the notation). 
For measure-preserving systems $(X, \mu, G$) and $(Y, \nu, G)$, we say that $(Y, \nu, G)$ is a {\em factor} of $(X, \mu, G)$ is there exists a measurable map $\pi\colon X\to Y$ such that  $\pi\circ g = g\circ \pi$ for all $g \in G$ almost everywhere and $\pi_{*}\mu = \nu$. As it is always clear from the context if we mean topological or measurable factor, we do not distinguish them in the terminology.

\subsection{Symbolic systems}
\label{sec:subshifts}
Many of our examples are symbolic systems.  Let $\CA$ be a finite set and let $\CA^\Z$ denote all functions $x\colon\Z\to\CA$. 
Writing $x\in \CA^\Z$ as $x = (x_i)_{i\in\Z}$, the space $\CA^\Z$ is a compact metric space with metric 
$$
d(x, y) = d\bigl((x_i)_{i\in\Z}, (y_i)_{i\in\Z}\bigr) = 
2^{-\inf \{|i|: x_i\neq y_i\}}.
$$
Define the {\em left shift} $\sigma\colon \CA^\Z\to\CA^\Z$ by setting
$(\sigma x)_i = x_{i+1}$ for all $i\in\Z$ and note that this is a homeomorphism.  Thus $(X, \sigma)$ is a system. 

When $\CA = \{0,\ldots,n-1\}$, the system 
$(X_{n} = \{0,\ldots,n-1\}^{\mathbb{Z}}, \sigma)$ is the {\em full shift on $n$ symbols}.  Endowing the set $X_{n}^{+} = \{0,\ldots,n-1\}^{\mathbb{N}}$ with the (non-invertible) shift $\sigma$, we obtain the {\em one-sided full shift on $n$ symbols}.  More generally, if $X\subset\CA^\Z$ is closed and  $\sigma$-invariant, then $(X, \sigma)$ is a {\em subshift}.  

For $w = w_{-k}\dots w_k\in \CA^{2k+1}$, set 
$$[w] = \{x\in \CA^\Z: x_i = w_i \text{ for } i=-k, \dots, k\}, 
$$
and we refer to such $[w]$ as a {\em cylinder set}. 
We write $[w]^+$ to be the cylinder set determined by the entries $w_1\dots w_{k}$.  
When $(X, \sigma)$ is a subshift, the {\em language} $\CL(X)$ is defined to be
$$
\CL(X) = \{w\in\CA^*: [w]^+\cap X\neq \emptyset\}, 
$$
and note that $\CL(X) = \bigcup_{k=1}^\infty\CL_k(X)$, where $\CL_k(X)$ denotes the words of length $k$ in $\CL(X)$. 

\subsection{Amenable groups} 
For a group $G$, a sequence $(F_n)_{n\in\N}$ of finite subsets of $G$ is a {\em F\o lner sequence} if for every $g \in G$ and $\delta > 0$, there is some $N\in\N$ such that for all $n\geq N$ we have $|gF_n\Delta F_n| < \delta |F_n|$, and the group $G$ is {\em amenable} if it admits a F\o lner sequence. The  F\o lner sequence 
$(F_n)_{n\in\N}$ is {\em tempered} if for some constant $C > 0$ and for all $n\in\N$, we have 
\[
\big\vert \bigcup_{k < n} F_k^{-1}F_n
\big\vert \leq C|F_n|.
\]
Note that any F\o{}lner sequence for a group admits a tempered F\o{}lner subsequence. Analogously, a semigroup is  {\em (left) amenable} if it admits a (left) F\o lner sequence.

We  make use of the following well-known facts for an amenable semigroup $G$. The Krylov-Bogoliubov Theorem 
 states that if $G$ acts continuously on a compact metric space $X$, then there exists a $G$-invariant Borel probability measure on $X$. Furthermore, if $G$ acts on $Y$ and $\pi \colon X \to Y$ is a factor map of the action, then the induced map on measures $\pi_{*} \colon \CM_{G}(X) \to \CM_{G}(Y)$ is surjective.

 \subsection{Joinings}
A {\em joining} of the systems $(X, \mu, G)$ and $(Y, \nu, G)$ is a measure $\lambda$ on $X\times Y$ which is invariant under the diagonal action $g \cdot (x,y) = (gx,gy)$, for which $\mu$ and $\nu$ respectively are the marginals under the natural projections onto each of the coordinates. 
 When $(X, \mu, G) = (Y, \nu, G)$, the joining $\lambda$ is a {\em self-joining of order $2$}, or just a {\em self-joining}, and a measure on the $k$-fold product of $(X, \mu, G)$ with itself that has $\mu$ as each of its marginals is a {\em $k$-fold self-joining}.

For a group $G$ and $g_1, \dots, g_k\in G$, we define the measure $\mu_{g_1, \dots, g_k}$ on $X^{k}$ by
$$\mu_{g_1, \dots, g_k}(A_1\times \dots A_k) = \mu(g_1A_1\cap \dots\cap g_kA_k).$$
This is a {\em $k$-fold off-diagonal self-joining} when $g_ig_j^{-1}$ lies in the center of $G$ for all $1\leq i,j\leq k$. A $k$-fold self-joining is a {\em product of off-diagonal self-joinings} if there exists a partition of $\{1, \dots, k\}$ such that each block is an off-diagonal self-joining and any two coordinates in distinct subsets are independent. The system has {\em minimal self-joinings of all orders} if every $k$-fold self-joining is a product of off-diagonal self-joinings.

\section{Invariant random compacts}\label{sec:ircs}
\subsection{The space of compact subsets}
Let $\CK(X)$ denote the space of nonempty compact subsets of $X$ endowed with the Hausdorff metric $d_H$, defined for $K_1, K_2\in \CK(X)$ by 
$$
d_H(K_1, K_2) = \max\{\sup_{x\in K_1}d(x, K_2), \sup_{y\in K_2}d(y, K_1) \}, 
$$
where 
$$d(x, K_2) = \inf_{y\in K_2}d(x,y) \quad 
\text{ and } \quad 
d(y, K_1) = \inf_{x\in K_1}d(y,x).
$$
The action of $G$ on $X$ induces an action of $G$ on $\CK(X)$ given by 
$$g \cdot A = g(A)$$
for $g\in G$, and this action is continuous in the Hausdorff metric. When $f \colon X \to X$ is a continuous map, we occasionally write $f^{\CK} \colon \CK(X) \to \CK(X)$ to denote the induced map.

The space $X$ embeds isometrically into $\CK(X)$ via the map $x \mapsto \{x\}$, sending each point $x\in X$ to the set consisting only of that point. 
More generally, for each $k \ge 1$,  there is a continuous map
$\rho_{k} \colon X^{k} \to \CK(X)$
defined by
\begin{equation}
    \label{def:rho-k}
\rho_{k}\bigl((x_{1},\ldots,x_{k}) \bigr) =\{x_{1},\ldots,x_{k}\}
\end{equation}
where $X^{k}$ denotes the product $\prod_{i=1}^{k}X$.
For $k\geq 2$, the map $\rho_k$ is not injective except when $X$ only consists of a single point.  
Letting $\CK_{\le k}(X)$ denote all subsets of size at most $k$, the image of $\rho_{k}$ is exactly $\CK_{\le k}(X)$, and is compact for each $k\geq 1$. We write $\CK_{\fin}(X) = \bigcup_{k=1}^{\infty}\CK_{\le k}(X)$ and $\CK_{\inf}(X) = \CK(X) \setminus \CK_{\fin}(X)$. The subset $\CK_{\fin}(X)$ is an $F_{\sigma}$-set, and hence Borel, and it is straightforward to check that it is dense in $\CK(X)$.

The diagonal action of $G$ on $X^{k}$ is given by
$$g \cdot (x_{1},\ldots,x_{k}) \mapsto (g \cdot x_{1},\ldots,g \cdot x_{k})$$
for $x_1, \dots, x_k\in X$ and $g\in G$. 
When $(X,G)$ is a system, unless otherwise stated, when we speak of a $G$-action on $X^{k}$ we  mean the diagonal action. For every $k \ge 1$,  the map $\rho_{k}$ is $G$-equivariant and hence is a factor map between the diagonal action of $G$ on $X^{k}$ and the induced $G$-action on $\CK_{\le k}(X)$.

At times it is convenient to work with the Vietoris topology on $\CK(X)$, which generates the same topology as the Hausdorff distance (see~\cite[4.F]{kechris}). A subbase for this topology consists of sets of the form $[U] = \{A \in \CK(X) : A \cap U \ne \emptyset\}$ and $\langle U \rangle = \{A \in \CK(X) : A \subset U\}$, where $U$ is any open set in $X$.

\subsection{Measures on $\CK(X)$}

We formulate the main definition for our setting. 
\begin{definition}
For a topological dynamical system $(X, G)$, an \emph{invariant random compact for the action of $G$ on $X$} is a Borel probability measure $\mu$ on $\CK(X)$ that is invariant under the induced action of $G$ on $\CK(X)$.  When the context is clear, we shorten this and refer to $\mu$ as an IRC for $(X,G)$.
\end{definition}
In other words, a Borel probability measure $\mu$ on $\CK(X)$ is an IRC if $g_*\mu = \mu$ for all $g\in G$, or equivalently that $\mu(g^{-1}(E)) = \mu(E)$ for all Borel sets $
E\subset \CK(X)$. 

For $Y\subset X$ compact, let $\delta_Y$ denote the Dirac measure (delta mass) concentrated on $Y$.  
Since we assume each $g \in G$ acts surjectively, the delta mass $\delta_X$ of the space $X$ is always an IRC for the the $G$-action on $X$. Thus we refer to an IRC that is not $\delta_X$ as \emph{nontrivial}.  
More generally, for any 
compact invariant set $Y\subset X$, the measure $\delta_Y$ is an IRC.

A conjugacy between systems $(X, G)$ and $(Y, G)$ induces a conjugacy of the corresponding systems $(\CK(X), G)$ and $(\CK(Y), G)$.  It follows that if the systems $(X,G)$ and $(Y,G)$ are topologically conjugate, then their spaces of IRCs are isomorphic. 

An IRC $\mu$ for an action of $G$ on $X$ has a canonical decomposition  $\mu = \mu_{F} + \mu_{I}$ into its finitary and  nonfinitary parts, where 
the measures $\mu_{F}, \mu_{I}$ on $\CK(X)$ are defined by
$$\mu_{F}(A) = \mu(A \cap \CK_{\fin}(X)) \qquad \text{ and } \qquad \mu_{I}(A) = \mu(A \cap \CK_{\inf}(X)).$$
The subset $\CK_{\fin}(X)$ is $G$-invariant, and
since we assume every $g \in G$ acts finite-to-one,
the subset $\CK_{\inf}(X)$ is $G$-invariant as well. In this case, both $\mu_F$ and $\mu_I$ are $G$-invariant and $\mu(\CK_{\inf}(X)) = \mu_{I}(\CK(X))$ and $\mu(\CK_{\fin}(X)) = \mu_{F}(\CK(X))$.

We say the invariant random compact $\mu = \mu_{F} + \mu_{I}$ is 
\emph{nonfinitary} if $\mu_{F} = 0$ and we say  $\mu$ is \emph{finitary} if $\mu = \mu_{F}$. 
Note that $\mu$ being nonfinitary is equivalent to  $\mu(\CK_{\fin}(X)) = 0$, and $\mu$ being finitary is equivalent to $\mu(\CK_{\fin}(X)) = 1$.

    It may seem natural to replace the condition that a measure $\mu$ be nonfinitary by the stronger condition that $\supp(\mu)\cap\CK_{\leq k}(X) = \emptyset$ for all $k\geq 1$, which clearly implies that $\mu(\CK_{\leq k}(X)) = 0$ for all $k\geq 1$.  However, these two conditions are not equivalent.  To check this, let $X$ be an infinite compact space, fix some nonisolated $x\in X$, and let $(B_n)_{n\in\N}$ be a sequence of closed neighborhoods of $x$ whose diameters shrink to $0$ as $n\to\infty$. Define a measure $\mu$ on $\CK(X)$ by setting $\mu = \sum_n\omega_n\delta_{B_n}$ for some weights $\omega_n > 0$ satisfying $\sum_n\omega_n = 1$.  Then $\mu$ assigns every finite subset of $X$ measure $0$ but  the singleton set $\{x\}$ lies in the support of $\mu$, and so the stronger condition does not hold.

For $k\geq 1$ and $\mu \in \CM_G(X)$, we write $\mu^{\otimes k}$ for the product measure $\mu \otimes \cdots \otimes \mu$ on $X^{k}$. 
If $\nu$ is a measure on $X^{k}$ that is invariant for the diagonal action of $G$, then the pushforward measure $(\rho_{k})_{*}(\nu)$ is a finitary IRC for $(X,G)$. Given $\mu \in \CM_G(X)$ and $k\geq 1$, we use  $\mu^{\CK}_{k}$ to denote the measure $(\rho_{k})_{*}(\mu^{\otimes k}) \in \CM_G(\CK(X))$.

\begin{proposition}\label{prop:purelyfinitaryIRCs}
If $G$ is amenable, then every finitary IRC is a pushforward of a measure on $\prod_{i=1}^{k}X$ invariant under the diagonal action.
\end{proposition}
\begin{proof} 
The map $\rho_{k} \colon \prod_{i=1}^{k}X \to \CK(X)$ defined in~\eqref{def:rho-k} is a factor map surjecting onto $\CK_{\le k}(X)$, so the pushforward map on measures is surjective onto the set of $G$-invariant Borel probability measures on $\CK_{\le k}(X)$.
\end{proof}

We note two basic properties, which follow immediately from the definitions.

\begin{proposition}
\label{prop:basics}
Suppose $(X,G)$ is a system.
\begin{enumerate}
\item
\label{item:basic-one}
A subset $Y \subset X$ is invariant if and only if $Y$ is a fixed point of the action of $G$ on $\CK(X)$. 
In particular, $(X,G)$ is minimal if and only if the only fixed point of $(\CK(X),G)$ is $X$.
\item
\label{item:basic-three}
The system $(X,G)$ is totally minimal if and only if the only periodic point of $(\CK(X),G)$ is $X$.
\end{enumerate}
\end{proposition}

It follows immediately that a totally minimal action has no atomic IRCs apart from $\delta_{X}$.

\begin{proposition}
\label{prop:pass-to-factor}
Suppose $G$ is amenable, $(X,G)$ and $(Y,G)$ are systems, and $\pi \colon X \to Y$ is a factor map. If $(Y,G)$ has nontrivial nonfinitary IRCs, then so does $(X,G)$.
\end{proposition}
\begin{proof}
The map $\pi \colon X \to Y$ induces a factor map 
$\pi^{\CK} \colon \CK(X) \to \CK(X)$ of the respective $G$-actions. Suppose $\mu$ is a nontrivial nonfinitary IRC for $(Y,G)$. Since $G$ is amenable, there exists $\nu \in \CM_{G}(\CK(X))$ such that $\pi^{\CK}_{*}(\nu) = \mu$. We have $(\pi^{\CK})^{-1}(\CK_{\inf}(Y)) \subset \CK_{\inf}(X)$ and $\mu(\CK_{\inf}(Y)) = 1$, so
$$\nu(\CK_{\inf}(X)) \ge \nu((\pi^{\CK})^{-1}(\CK_{\inf}(Y))) = \pi^{\CK}_{*}(\nu)(\CK_{\inf}(Y)) = \mu(\CK_{\inf}(Y)) = 1$$
and $\nu$ is nonfinitary. Since $\mu$ is nontrivial, we have $\mu(\{Y\}) < 1$. As $X \in (\pi^{\CK})^{-1}(\{Y\})$, it follows that
$$\nu(\{X\}) \le \nu((\pi^{\CK})^{-1}(\{Y\})) = \pi^{\CK}_{*}(\nu)(\{Y\})=\mu(\{Y\})<1$$
and so $\nu$ is also nontrivial. 
\end{proof}

\begin{proposition}
\label{prop:act-by-isometry}
If $G$ is an amenable group acting by isometries on an infinite compact space $X$, then $(X,G)$ has nontrivial nonfinitary IRCs.
\end{proposition}

\begin{proof}
Let $A$ be a proper infinite compact subset of $X$. Then there exists some $\delta > 0$ such that $d_{H}(g(A),X) = d_{H}(A,X) > \delta$ for every $g \in G$. Let $(F_{n})_{n\in\N}$ be a F\o{}lner sequence for $G$ and let $\mu$ be any weak* limit of $\frac{1}{|F_{n}|}\sum_{g \in F_{n}}\delta_{g(A)}$. Since $d_{H}(g(A),X) > \delta$ for all $g \in G$, it follows that $\mu$ is nontrivial.   Moreover, given $m \ge 1$ there exists $\varepsilon_{m}> 0$ such that $d_{H}(A,\CK_{\le m}(X)) \ge \varepsilon_{m}$ and hence $d_{H}(g(A),\CK_{\le m}(X)) \ge \varepsilon_{m}$ for all $g \in G$ as well. It follows that $\mu(\CK_{\le m}(X)) = 0$ for every $m \ge 1$, and hence $\mu$ is nonfinitary. 
\end{proof}

For example, suppose $T \colon \mathcal{O}_{n} \to \mathcal{O}_{n}$ is an $n$-adic odometer. Then $\mathcal{O}_{n}$ is a Cantor set and $\CK(\mathcal{O}_{n})$ is also a Cantor set (see~\cite[Section 8]{choquet}). Moreover, the induced action of $T$ on $\CK(\mathcal{O}_{n})$ is also an isometry. Since any isometry of a Cantor system is a union of its minimal components~\cite[Corollary 10, Chapter 1 and Theorem 2, Chapter 2]{Auslander} and each minimal component is, up to topological conjugacy, an odometer, the induced action of $T$ on $\CK(\mathcal{O}_{n})$ decomposes into odometers. Thus in this case, every IRC is a combination of Haar measures on these odometers.

The existence of nonfinitary invariant random compacts for a system $(X, G)$, roughly speaking, concerns the 
existence of infinite compact sets $Y \subset X$
and frequencies of elements $g$ for which $gY$ becomes either $\varepsilon$-dense or $\varepsilon$-close to finite sets.  
More precisely, for an amenable group $G$ with F\o lner sequence $(F_n)_{n\in\N}$ acting on $X$, infinite compact set $Y\subset X$, and $\varepsilon > 0$, define  
\begin{equation}
    \label{def:Dn}
    \mathcal{D}_n(Y, \varepsilon) = 
    \frac{1}{|F_n|}\big\vert
    \{g\in F_n: d(gY, X) \geq  \varepsilon \}
    \big\vert.
\end{equation}
We also define for each $r \ge 1$
\begin{equation}
    \label{def:En}
    \mathcal{E}_n(Y, r, \varepsilon) = 
    \frac{1}{|F_n|}\big\vert
    \{g\in F_n: d(gY, \CK_{\le r}(X)) \geq  \varepsilon \}
    \big\vert.
\end{equation}

\begin{proposition}
\label{prop:amenable-stretch}
Suppose an amenable group $G$ acts on $X$. If the system $(X,G)$ admits a nontrivial nonfinitary IRC, then for every $0 < c \le 1$ there exists an infinite compact set $Y \subset X$ and $\varepsilon_{r} > 0$ for $1 \le r \le \infty$ satisfying
$$\lim_{n \to \infty} \mathcal{E}_{n}(Y,r,\varepsilon_{r}) > 1-c\quad \text{ and } \quad  \lim_{n \to \infty} \mathcal{D}_{n}(Y,\varepsilon_{\infty}) > 0.$$
\end{proposition}
\begin{proof}
Let $\mu$ be a nontrivial nonfinitary IRC for $(X,G)$ and $0 < c \le 1$. Since $\mu(\{X\}) < 1$ and $\mu(\CK_{\le r}(X)) = 0$ for every $1 \le r < \infty$, there exists $\varepsilon_{r} > 0$ such that
$$\mu(\CK(X)\setminus B_{\varepsilon_{\infty}}(X)> 0 \quad \text{ and } \quad \mu(\CK(X) \setminus B_{\varepsilon_{r}}(\CK_{\le r}(X))) > 1-c.$$
Letting $(F_n)_{n\in\N}$ be a tempered F\o lner sequence, the Pointwise Ergodic Theorem~\cite{lindenstrauss} gives the existence of an infinite compact set $Y\subset X$ satisfying
\begin{equation*}
    \lim_{n \to \infty} \mathcal{E}_{n}(Y,r,\varepsilon_{r}) > 1-c
 \text{ for every } r \ge 1 \quad \text{ and } \quad 
\lim_{n \to \infty} \mathcal{D}_{n}(Y,\varepsilon_{\infty}) > 0. \qedhere
\end{equation*} 
\end{proof}

A natural way to try to construct nontrivial nonfinitary IRCs is by starting with an infinite compact set $Y$ and considering weak* limits of $\frac{1}{|F_{m}|}\sum_{g \in F_{m}}\delta_{gY}$. Interesting cases where such a process can only produce measures with nontrivial finitary component can be seen in examples due to Berend and Peres~\cite{Ber-Per} for the multiplicative semigroup action of $\mathbb{N}^{2}$ given by $\times 2$ and $\times 6$ on the torus $\mathbb{T}$. For instance, letting $Y = \{2^{-r(m)}\}_{m \in \mathbb{N}} \cup \{0\}$ where $r(m)$ is a rapidly growing sequence (such as $r(m) = 2^m$), then the only accumulation points of $Y$ in $\CK(\mathbb{T})$ are subsets of the form $\{0,x\}$ for $x \in \mathbb{T}$.

\subsection{IC-rigidity}

\begin{definition}
The action of $G$ on $X$ is \emph{infinite compact rigid} if the only nonfinitary IRC is $\delta_{X}$, and we refer to such an action as  \emph{IC-rigid}. 
\end{definition}

In other words, an action of $G$ on $X$ is IC-rigid if and only if for every IRC $\mu$, the decomposition into of $\mu$ into its finitary and infinitary parts has the form $\mu = \mu_{F} + c \cdot \delta_{X}$ for some $c \ge 0$.  
It follows immediately from Proposition~\ref{prop:pass-to-factor} that for actions of amenable groups, every factor of an IC-rigid system is also IC-rigid.  
We defer examples of IC-rigid actions until  Section~\ref{sec:ICridigexamples}, after we have developed further tools for building such actions. 

Recall that the action of $G$ on $X$ is {\em almost minimal} if every proper compact $G$-invariant set is finite.
\begin{proposition}\label{prop:icrigidweakmixing}
If $(X,G)$ is IC-rigid, then $(X,G)$ is almost minimal. If in addition $G$ is amenable, then $(X,G)$ has trivial maximal equicontinuous factor. In particular, if $G$ is abelian and $(X,G)$ is minimal, then $(X,G)$ is topologically weak mixing.
\end{proposition}
\begin{proof}
The first part follows from Part~\ref{item:basic-one} of Proposition~\ref{prop:basics}. If $(X,G)$ has a nontrivial maximal equicontinuous factor, then it has a nontrivial factor onto an isometry (any equicontinuous action can be made into an isometric action by changing to an equivalent metric; see~\cite[Chapter 2]{Auslander}). Combining Propositions~\ref{prop:pass-to-factor} and~\ref{prop:act-by-isometry}, it follows that $(X,G)$ has a nontrivial nonfinitary IRC. For $(X,G)$ minimal with abelian $G$, topologically weak mixing is equivalent to trivial maximal equicontinuous factor (see~\cite[Chapter 9]{Auslander}). 
\end{proof}

As a consequence, a distal group action on an infinite compact space always has nontrivial nonfinitary IRCs.

\subsection{Examples of IRCs}
\label{sec:examples-of-IRC}
We start with a simple example of a nonfinitary and nonatomic IRC that does not arise from isometries. 
\begin{example}[Skew extensions]
Suppose $\varphi \colon \mathbb{T} \to \mathbb{R}$ is a continuous function, $\alpha \in \mathbb{R}$ is irrational, and consider the skew extension $\Phi \colon \mathbb{T}^{2} \to \mathbb{T}^{2}$ defined by $\Phi(x,y) = (x+\alpha,y+\varphi(x))$. The map $f \colon \mathbb{T} \to \CK(\mathbb{T}^{2})$ defined by $f(x) = \{(x,s) : 0 \le s \le 1\}$ is continuous and $f \circ T_{\alpha} = \Phi \circ f$ where $T_{\alpha}\colon \T\to\T$ is the rotation by $\alpha$. 
The pushforward $f_{*}(\lambda)$ of the Lebesgue measure $\lambda$ is a nonatomic nonfinitary IRC for $(\mathbb{T}^{2},\Phi)$.
\end{example}  

Even for systems with simple dynamics, the space of IRCs can be rich.
\begin{example}[Sunny-side up shift]
   The sunny-side up shift over the alphabet $\{0,1\}$ is defined by setting ${X = \{ x \in \{0,1\}^{\mathbb{Z}} : x \textrm{ has at most a single } 1 \}}$ and taking  $G = \mathbb{Z}$ to be the action generated by the shift $\sigma$  on $X$. Given $i \in \mathbb{Z}$, let $x_{i}$ denote the point with a $1$ at position $i$ and $0$ at all other entries and let $x_{e}$ denote the fixed point of all $0$s. For $y \in X$, let $\supp(y) = \{i \in \mathbb{Z} : y_{i} = 1\}$.
    
   Let  $F \colon \{0,1\}^{\mathbb{Z}} \to \CK(X)$ be the function defined 
    by mapping $y \in \{0,1\}^\Z$ to 
    $A_y\in\CK(X)$, where 
\[
A_{y} =
\begin{cases} 
      \{x_{e}\} & \textrm{ if }  y = 0^{\infty}\\
      \{x_{i} : i \in \supp(y)\}  & \textrm{ if supp}(y) \textrm{  is finite } \\
       \{x_{i} : i \in \supp(y)\} \cup \{x_{e}\} & \textrm{ if } \supp(y) \textrm{ is infinite}.
   \end{cases}
\]

It is easy to check that $F$ is continuous. 
Moreover, $F$ is equivariant with respect to the shift actions on $\{0,1\}^{\mathbb{Z}}$ and $\CK(X)$ respectively. 

The map $F$ is close to being a bijection: it is injective, and its image contains all of $\CK(X)$ other than finite subsets that contain the point $x_{e}$.  However, we can extend $F$ to a bijection by extending $\{0,1\}^{\mathbb{Z}}$ to a space $Y$ 
defined as follows: take $\{0,1\}^{\mathbb{Z}}$ and replace each point $y$ which has finitely many $1$s by a pair $\tilde{y} = \{y_{1},y_{2}\}$. Define the topology on $Y$ by the usual topology in the $\{0,1\}^{\mathbb{Z}}$ direction and further having $d(y_{1},y_{2}) \to 0$ if $y \to 0^{\infty}$. The shift map $\sigma$ on $\{0,1\}^{\mathbb{Z}}$ extends to a homeomorphism $\tilde{\sigma}$ on $Y$ by $\tilde{\sigma}(y_{i}) = \sigma(y)_{i}$. 
We then extend $F$ by setting $F(y_{1}) = F(y)$ and $F(y_{2}) = F(y) \cup \{x_{e}\}$. This extended $\tilde{F}$ defines a topological conjugacy between $\tilde{\sigma}$ on $Y$ and the map induced by the shift map on $\CK(X)$.

It follows that the pushforward via $\tilde{F}$ of any shift-invariant Borel probability measure on $\{0,1\}^{\mathbb{Z}}$ gives rise to an invariant random compact for the shift on $X$. Note that the pushforward of the delta mass at the fixed point of all $0$s is  $\delta_{x_{e}}$ in $\CM_G(\CK(X))$, 
and the pushforward of the delta mass at the fixed points of all $1$s is $\delta_{X}$.
\end{example}

A similar construction works for any heteroclinic point. It follows that if a $\mathbb{Z}$-action $T \colon X \to X$ contains either a homoclinic or heteroclinic point, then the system $(X,T)$ contains an abundance of nontrivial nonfinitary IRCs.

\begin{example}[Irrational rotation]
Let $\alpha \in \mathbb{R} \setminus \mathbb{Q}$ and $T_{\alpha} \colon \mathbb{T} \to \mathbb{T}$ by $T_{\alpha}(x) = x + \alpha  \mod 1$ be the associated irrational rotation. Then every IRC for $(\mathbb{T},T_{\alpha})$ is a combination of Haar measures on tori. For instance, given $0 < c < 1$, one can check that the action of $T_{\alpha}^{\CK}$ on the orbit closure of the interval $[0,c]$ in $\CK(\mathbb{T})$ under $T_{\alpha}^{\CK}$ is topologically conjugate to $T_{\alpha}$ on $\mathbb{T}$.  
More generally, $T_{\alpha}$ induces an isometry $\CK(T_{\alpha}) \colon \CK(\mathbb{T}) \to \CK(\mathbb{T})$. This isometry decomposes as a disjoint union of its minimal components $\{\mathcal{D}_{i}\}_{i \in I}$ by~\cite[Corollary 10, Chapter 1 and Theorem 2, Chapter 2]{Auslander}. 

We claim that each of these components $\mathcal{D}_{i}$ is topologically conjugate to rotation by $\alpha$ on $\mathbb{T}$. Given such a component $\mathcal{D} \subset \CK(\mathbb{T})$, let $Y \in \mathcal{D}$. To define a map $\Psi \colon \mathcal{D} \to \mathbb{T}$, we first fix a point $y \in Y$ and set $\Psi (T_{\alpha}^{\CK})^{m}(Y) = T^{m}_{\alpha}(y)$. 
Note then that $\Psi \circ (T_{\alpha}^{\CK})^{m}(Y) = T^{m}_{\alpha} \circ \Psi(Y)$. Then $\Psi$ is uniformly continuous on the dense orbit $\{(T_{\alpha}^{\CK})^{m}(Y)\}_{m \in \mathbb{Z}}$ and extends to a continuous function $\tilde{\Psi} \colon \mathcal{D} \to \mathbb{T}$ satisfying $\tilde{\Psi} \circ T_{\alpha}^{\CK} = T_{\alpha} \circ \tilde{\Psi}$. To see that $\tilde{\Psi}$ is surjective, given $z \in \mathbb{T}$, there exists $n_{k}$ such that $T_{\alpha}^{n_{k}}(y) \to z$. By compactness there exists an accumulation point $Z$ of $(T_{\alpha}^{\CK})^{n_{k}}(Y)$ in $\mathcal{D}$, and then $\tilde{\Psi}(Z) = z$. To see that $\tilde{\Psi}$ is injective, suppose $z_{1} = \tilde{\Psi}(Z_{1}) = \tilde{\Psi}(Z_{2}) = z_{2}$ for some $Z_{1},Z_{2} \in \mathcal{D}$. By minimality, choose $m_{k}, n_{k}$ such that $(T^{\CK}_{\alpha})^{m_{k}}(Y) \to Z_{1}$ and $(T_{\alpha}^{\CK})^{n_{k}}(Y) \to Z_{2}$. Then $T_{\alpha}^{m_{k}}(y) \to z_{1}$ and $T_{\alpha}^{n_{k}}(y) \to z_{2}$. Since $T_{\alpha}$ is an isometry, this implies $Z_{1} = Z_{2}$, and $\tilde{\Psi}$ is injective.
\end{example}

\subsection{Fiber IRCs}
Invariant random compacts may also be constructed from factors of an action.
\begin{proposition}
\label{prop:factors-pass}
Suppose $\pi \colon X \to Y$ is a factor map between systems $(X,G)$ and $(Y,G)$. Then the map $\tilde{\pi} \colon Y \to \CK(X)$ defined by $\tilde{\pi}(y) = \pi^{-1}(y)$ is Borel and is equivariant for the actions of $G$ on $Y$ and $\CK(X)$, respectively.
\end{proposition}
\begin{proof}
We argue in terms of the Vietoris topology on $\CK(X)$. Fix an open set $U$ in $X$. Then
$$\tilde{\pi}^{-1}([U]) = \{y \in Y : \pi^{-1}(y) \cap U \ne \emptyset\} = \pi(U)$$
is an $F_{\sigma}$ set since $X$ is compact. As well, we have
$$\tilde{\pi}^{-1}(\langle U \rangle) = \{y \in Y : \pi^{-1}(y) \subset U\} = \{y \in Y : y \not \in \pi(X \setminus U)\} = Y \setminus \pi(X \setminus U),$$
which is open. Since both types are Borel and these form a subbase for the topology, the result follows. 
\end{proof}
Suppose that $\pi \colon X \to Y$ is a factor map between $(X,G)$ and $(Y,G)$ and let $\tilde{\pi} \colon Y \to \CK(X)$ be the map defined in Proposition~\ref{prop:factors-pass}. If $\mu$ is a $G$-invariant Borel probability measure on $Y$, then the pushforward measure $\nu = \tilde{\pi}_{*}(\mu)$ is an IRC for $G$ acting on $X$, and  we refer to it as the {\em associated fiber IRC}.

Fiber IRCs allow us to construct an abundance of examples of nonatomic nonfinitary IRCs, for instance by taking skew-products with infinite fibers over a base $Y$ which has a nonatomic $G$-invariant Borel probability measure.

We also use this to understand the constraints on fibers of IC-rigid actions. 
\begin{theorem}
\label{th:finite-fiber}
Suppose $G$ acts on $X$ and the action is IC-rigid. If $(Y,G)$ is a nontrivial factor of $(X,G)$, then for every $G$-invariant Borel probability measure $\mu$ on $Y$, $\mu$-almost every fiber is finite.
\end{theorem}
\begin{proof}
Let $\mu$ be a $G$-invariant Borel probability measure on $Y$ and let $\nu = \tilde{\pi}_{*}(\mu)$ be the associated fiber IRC.  By assumption, the decomposition of $\nu$ into its finitary and infinitary parts has the form $\nu = \nu_{F} + c \cdot \delta_{X}$ for some $c \ge 0$. However $\pi^{-1}(y) \ne X$ for every $y \in Y$ since we  assume that $Y$ is not a point. It follows that $c=0$, and so $\nu = \nu_{F}$ is finitary, giving the result.
\end{proof}

\subsection{Measure-theoretic rigidity implies nontrivial nonfinitary IRCs}
For a group $G$, 
we write $g_{m} \to \infty$ to mean that the sequence $g_{m}$ eventually leaves every finite set in $G$.   
We recall that a system $(X,G,\mu)$ is {\em measure-theoretically rigid} if there exists a sequence $g_m\in G$ with $g_{m} \to \infty$ such that 
$$
\mu(g_m^{-1}A\Delta A) \to 0
$$
for every measurable set $A\subset X$. 
The system is {\em $\mu$-essentially free} if for all $g \ne \Id$, we have $\mu(\Fix(g))=0$.

\begin{theorem}\label{thm:rigiditygivesIRCs}
Suppose $G$ is a countable group acting on a space $X$ and suppose there exists $\mu \in \CM_G(X)$ for which $(X,G,\mu)$ is $\mu$-essentially free and measure-theoretically rigid. Then $(X,G)$ is not IC-rigid.
\end{theorem}

We make use of the following standard lemma in the proof of theorem, and we include a proof for completeness.  

\begin{lemma}
\label{lemma:full-measure}
Suppose $(X,G,\mu)$ is measure-theoretically rigid. Then there exists a sequence $g_{m} \to \infty$ in $G$ such that $g_{m}x \to x$ as $m \to \infty$ for $\mu$-almost every $x\in X$.
\end{lemma}

\begin{proof}
The system $(X, G, \mu)$ is rigid if and only if for all $\varepsilon > 0$, there exists infinitely many $g = g(\varepsilon)$ such that 
\begin{equation}
    \label{eq:choose-sequence}
\mu(\{x\in X: d(gx, x)< \varepsilon \}) > 1-\varepsilon
\end{equation}
(see for instance~\cite[Proposition 3.1]{Huangplus}; the proof there is written for $G=\mathbb{Z}$ but holds in general, and the fact there are infinitely many is clear from the proof). For each $k\geq 1$, let $g_k\in G$ satisfy~\eqref{eq:choose-sequence} for $\varepsilon = 1/k$, and without loss of generality we may assume the $g_{k}$ are pairwise distinct.  Setting $A_{k} = \{x \in X : d(g_{k}x,x) < \frac{1}{k}\}$, we have that $\mu(X \setminus A_{k})\le 1/k$.

Set $B_{j} = A_{2^{j}}$ so that $\mu(X \setminus B_{j}) \le \frac{1}{2^{j}}$. Then $\sum_{j}\mu(X \setminus B_{j})$ is finite, and so  Borel-Cantelli implies 
$$\mu\left(\bigcap_{j=1}^{\infty} \bigcup_{k=j}^{\infty}X \setminus B_{j}\right) = 0.$$
Thus for $\mu$-almost every  $x\in X$,  there exists $J(x)$ such that $x \in X \setminus B_{j}$ for all $j \ge J(x)$, and hence $g_{2^{j}}x \to x$.
\end{proof}

\begin{proof}[Proof of Theorem~\ref{thm:rigiditygivesIRCs}]
By Lemma~\ref{lemma:full-measure}, there is a sequence $(g_m)_{m\in\N}$ of distinct elements satisfying $g_{m}x \to x$ as $m \to \infty$ for $\mu$-almost every $x \in X$. For $x \in X$, set $E(x) = \{x,g_{1}(x),g_{2}(x),\ldots\} \subset X$. Define $f \colon X \to \CK(X)$ by setting $f(x) = \overline{E(x)}$, the closure of $E(x)$.

We claim the map $f$ is Borel. We again use the Vietoris topology on $\CK(X)$. Let $U$ be an open set in $X$, nonempty and proper in $X$. First consider $f^{-1}([U]) = \{x \in X : \overline{E}(x) \cap U \ne \emptyset\}$. Since $U$ is open, $\overline{E}(x) \cap U \ne \emptyset$ if and only if there exists $m\in\N$ such that $g_{m}(x) \in U$, and so $f^{-1}([U]) = \bigcup_{m}g_{m}^{-1}(U)$ is open. Next consider $f^{-1}(\langle U \rangle) = \{x \in X : \overline{E}(x) \subset U\}$. Let $U_{k} = \{x \in U : d(x,X \setminus U) \ge \frac{1}{k}\}$. Then $\bigcup_{k=1}^{\infty}U_{k} = U$ and since $\overline{E}(x)$ is compact and $X \setminus U$ is compact, we have $\overline{E}(x) \subset U$ if and only if $\overline{E}(x) \subset U_{k}$ for some $k\geq 1$. Thus $f^{-1}(\langle U \rangle) = \bigcup_{k=1}^{\infty}\bigcap_{m=1}^{\infty} g_{m}^{-1}(U_{k})$, which is Borel.

Define $\nu = f_{*}(\mu)$. We show that $\nu$ is a nonfinitary and nontrivial IRC.

Let $R = \{x \in X : g_{m}(x) \to  x \text{ as } m\to\infty\}$. Then $f(x) = E(x)$ 
for $x \in R$ and by Lemma~\ref{lemma:full-measure}, $\mu(R) = 1$. For each $j \ge 2$, define $\xi_{j} \colon X \to X^{j}$ by setting $\xi_{j}(x) = (x,g_{1}(x),\ldots,g_{j-1}(x))$ and consider the associated off-diagonal measure $\eta_{j} = (\xi_{j})_{*}(\mu) \in \CM_{G}(X^{j})$. 
This gives rise to a measure $\eta_{j}^{\CK} = (\rho_{j})_{*}(\eta_{j}) \in \CM_{G}(\CK_{\le j}(X))$.

We claim that the measures $\eta_{j}^{\CK}$ converge weak* to $\nu$. Define $S_{j} \colon X \to \CK(X)$ by $S_{j}(x) = \{x,g_{1}(x),\ldots,g_{j-1}(x)\}$. Then $\rho_{j} \xi_{j} = S_{j}$ so $(S_{j})_{*}(\mu) = \eta_{j}^{\CK}$. If $x \in R$, then $E(x)$ is compact, and $S_{j}(x) \to E(x)$ in $\CK(X)$ as $j \to \infty$. Applying Lemma~\ref{lemma:full-measure} again, this holds for $\mu$-almost every $x \in X$. Since $\eta_{j}^{\CK} = (S_{j})_{*}(\mu)$, it follows that the measures $\eta_{j}^{\CK}$ converge weak* to $\nu$, proving the claim.  

To show that $\nu$ is nonfinitary, let $k \ge 1$ and $j > k$, and set $A = \rho_{j}^{-1}(\CK_{\le k}(X))$. Then $A$ is the set of points in $X^{j}$ with at most $k$ distinct coordinates, and $B = \xi_{j}^{-1}(A) = \{x \in X : \{x,g_{1}(x),\ldots,g_{j-1}(x)\} \textrm{ has at most } k \textrm{ points.}\}$. 
If $x \in B$ then since $j \ge k$, there must exist $0 \le i_{1} < i_{2} \le j-1$ such that $g_{i_{1}}(x) = g_{i_{2}}(x)$, where we take $g_{0} = \Id$. 
It follows that $B \subset \bigcup_{0 \le i_{1} < i_{2} \le j-1}\Fix(g_{i_{1}}^{-1}g_{i_{2}})$. 
Since the $g_{j}$ are all distinct, by the $\mu$-essentially free assumption we have $\mu\left(\bigcup_{0 \le i_{1} < i_{2} \le j-1}\Fix(g_{i_{1}}^{-1}g_{i_{2}})\right)=0$, and so $\mu(B) = 0$ and $\eta_{j}(A) = 0$. 
Thus $\eta_{j}^{\CK}(\CK_{\le k}(X)) = 0$, and $\eta^{\CK}_{m}(\CK_{\le k}(X)) = 0$ for all $m > k$, and hence $\nu(\CK_{\le k}(X)) = 0$. Since $k$ is arbitrary, we have that  $\nu(\CK_{\fin}(X)) = 0$.

To check that $\mu$ is nontrivial, note that $\nu$-almost every $Y \in \CK(X)$ has the form $E(x) = \{x,g_{1}(x),g_{2}(x),\ldots\}$ for $x \in R$, and such $E(x)$ are compact proper subsets of $X$. 
\end{proof}
In~\cite{PavSch}, it is shown that a generic subshift in the space of infinite transitive subshifts of any full shift is minimal, uniquely ergodic, and measure-theoretically rigid. Combining this with Theorem~\ref{thm:rigiditygivesIRCs}, it follows that a generic infinite transitive subshift has nontrivial nonfinitary IRCs.  
Similar reasoning can be used for other classes. For example, a generic interval exchange transformation is minimal, and by a theorem of Veech~\cite[Theorem 1.3]{veech} also rigid, and hence also possesses nontrivial nonfinitary IRCs.

\section{Weakening IC-rigidity}
\label{sec:weakening}
\subsection{Weakly IC-rigid actions}
\label{sec:weakly-rigid}

For a space $X$, let $\CK_\ct(X)$ denote the collection of all countable subsets of $\CK(X)$ and $\CK_{\textrm{uc}}(X))$ denote the collection of all uncountable subsets of $\CK(X)$. 
Clearly $\CK_{\textrm{uc}}(X))$ is the complement of $\CK_{\ct}(X)$ in $\CK(X)$. A theorem of Hurewicz (see~\cite[Theorem, 27.5]{kechris}) shows that if $X$ is uncountable, then $\CK_{\textrm{uc}}(X)$ is $\Sigma_{1}^{1}$-complete, and hence not Borel. Since it is analytic however, it is measurable with respect to any complete measure, and so we can pass to a completion to ensure that the subsets $\CK_{\textrm{uc}}$ and $\CK_\ct$ are measurable.  
Namely, given a Borel measure $\mu$,  let $\tilde{\mathcal{B}}$ denote the the $\mu$-completion of the Borel $\sigma$-algebra $\mathcal{B}$, meaning that 
$$\tilde{\mathcal{B}} = \{A \cup Z : A \in \mathcal{B}, Z \subset N \in \mathcal{B} \textrm{ where } \mu(N) = 0\}.$$
Letting $\tilde{\mu}$ denote the corresponding completion of $\mu$ defined on $\tilde{\mathcal{B}}$ by $\tilde{\mu}(A \cup Z) = \mu(A)$, it follows that  if $\mu$ is a Borel  IRC defined on $\CK(X)$, then $\CK_{\ct}(X) \in \tilde{\mathcal{B}}$ and $\tilde{\mu}(\CK_{\ct}(X))$ is well-defined.

\begin{definition}
The action of $G$ on $X$ is {\em weakly IC-rigid} if 
the only IRC $\mu$ on $\CK(X)$ satisfying $\tilde{\mu}(\CK_\ct(X)) = 0$ is $\delta_X$.  
\end{definition}

It follows immediately from the definitions that if an action is IC-rigid then it is weakly IC-rigid. 
If in addition  $X$ is countable, then the two notions are equivalent. However, in general the two notions are distinct.

Say a system $(X,G,\nu)$ has {\em countable exceptions for products} if for every $x_{1} \in X$ and $r \ge 2$, the set of tuples $(x_{1},x_{2},\ldots,x_{r})$ which are not generic for $\nu^{\otimes r}$ is countable.

The following gives a sufficient condition for a system to be weakly IC-rigid.

\begin{theorem}\label{thm:msjweaklyicrigid}
Let $G$ be a countable amenable group acting on $X$. If $(X, G)$ is minimal, uniquely ergodic, and has countable exceptions for products with respect to the unique invariant measure, then it is weakly IC-rigid. 
\end{theorem} 

By~\cite{DRS}, the Chacon system is an invertible subshift that is minimal, uniquely ergodic, and weakly mixing (and also has minimal self-joinings of all orders with respect to  the unique invariant measure). 
Furthermore, 
by~\cite[Theorem 2]{DK}, the Chacon system has countable exceptions for products.  
It follows that the Chacon system is weakly IC-rigid. In Section~\ref{sec:Chacon-not-rigid}, we show that it is not IC-rigid.

\begin{proof}
Suppose $\mu$ is an IRC; we assume that $\mu$ is complete by replacing it with its completion $\tilde{\mu}$ if necessary, and by a slight abuse of notation we also denote this measure by $\mu$. 
Suppose that $\mu(\CK_\ct(X)) = 0$.  
    The set 
    $$
    \{(K, m)\in \CK(X)\times \CM(X): m(K) = 1 \text{ and } m \text{ is nonatomic}\}
    $$
    is a Borel 
    subset of $\CK(X)\times\CM(X)$ since $\{(K,m) \in \CK(X) \times \mathcal{M}(X) : m(k)=1\}$ 
    is a closed subset of $\CK(X) \times \mathcal{M}(X)$ and $\{(K,m) \in \CK(X) \times \mathcal{M}(X) : m \textrm{ is nonatomic}\}$ is a $G_{\delta}$ set. Furthermore, the projection onto the first coordinate of this  set is surjective onto $\CK_{\textrm{uc}}(X)$, since every uncountable compact space supports a nonatomic measure. Thus for $\mu$-almost every $K\subset\CK(X)$, we can make a measurable selection (for instance using the Jankov-von Neumann Theorem) and choose a nonatomic probability measure $m_K$ supported on $K$. 
    In particular, the assumption $\mu(K_{\ct}(X)) = 0$ implies for $\mu$-almost $K \in \CK(X)$, we can choose such a measure $m_{K}$.
    
We define a measure that uses $\mu$ to sample $K\in \CK(X)$ and, conditional on $K$, samples some independent identically distributed (i.i.d.) variables. Consider then the measure $\rho_0$ on $\CK(X)
\times X^\N$ 
defined by
\begin{equation}
    \label{eq:rho-0}
\rho_0 = \int_{\CK(X)} (\delta_K\otimes m_K\otimes m_K\otimes\dots)\, d\mu(K).
\end{equation}
By construction, $\rho_0$ is supported on the set 
\begin{equation}
    \label{eq:support-rho}
\{(K, (x_i)_{i=1}^\infty): x_i\in K\},
\end{equation}
the first coordinate has $\mu$ as it marginal, $x_i\in K$ almost surely, and since $m_K$ is nonatomic the coordinates of each $m_K$ are almost surely distinct.  

Consider $\CK(X) \times X\times X\times \dots$ with the product action of $G$, meaning 
$$(A,x_{1},x_{2},\ldots) \mapsto (g(A),g(x_{1}),g(x_{2}),\ldots)$$
for $g \in G$. 
Since $G$ is amenable, fix some F\o{}lner sequence $(F_{n})_{n\in\N}$ for $G$ and for each $n\in\N$,  define
$$
\rho_n = \frac{1}{|F_{n}|}\sum_{g \in F_{n}}g_{*}\rho_0.
$$
Passing to a weak* limit, we obtain a $G$-invariant measure $\rho$. Note that the marginal of $\rho$ on the first coordinate is $\mu$, since $\mu$ is $G$-invariant and the marginal of the first coordinate of each $\rho_{n}$ is also $\mu$.

Let $\nu$ denote the unique $G$-invariant Borel probability measure on $X$. For $r\in\N$, define $\pi_r\colon K(X)\times X^\N\to X^r$ to be the projection to coordinates $2$ through $r+1$, meaning that 
$\pi_r(K, x_1, \dots, x_r) = (x_1, \dots, x_r)$. Define $\mu_{n, r} = (\pi_r)_*\rho_n$. We claim that $(\pi_r)_*\rho = \nu^{\otimes r}$. To check this, if $\phi$ is a continuous function on $X^r$, then  
$$\int_{X^r}\phi\,d\mu_{n,r} = 
\int_{\CK(X)\times X^\N}\frac{1}{|F_n|}\sum_{g\in F_n}\phi(gx_1, \dots, gx_r)\, d\rho_0(K, x_1, x_2, \dots).
$$
Since $(X,G,\nu)$ has countable exceptions for products, for every $x_{1} \in X$, the set of $(x_{2},\ldots,x_{r})$ such that $(x_{1},\ldots,x_{r})$ is not generic for $\nu^{\otimes r}$, is a countable set. Since for $\mu$-almost every $K$, the measure $m_{K}$ used to define $\rho_{0}$ is nonatomic, it follows that for $\rho_{0}$-almost every tuple $(K,x_{1},\ldots)$, the tuple $(x_{1},\ldots,x_{r})$ is generic for $\nu^{\otimes r}$. By the Ergodic Theorem, it follows that for such tuples we have 
$$\frac{1}{|F_n|}\sum_{g\in F_n}\phi(gx_1, \dots, gx_r) \to \int_{X^{r}}\phi d \nu^{\otimes r}.$$
By the Dominated Convergence Theorem, it follows that 
$$\int_{\CK(X)\times X^\N}\frac{1}{|F_n|}\sum_{g\in F_n}\phi(gx_1, \dots, gx_r)\,d\rho_0  
\to \int_{\CK(X)}
\int_{X^r}\phi\, d\nu^{\otimes r}d \mu = \int_{X^r}\phi\, d\nu^{\otimes r}
$$
as $n \to \infty$, and hence the $r$-coordinate marginals $\mu_{n, r}$ converge weak* to $\nu^{\otimes r}$.  Furthermore, this holds for all $r\geq 1$.

Considering the subsequence $\rho_{n_k}\to \rho$ in the weak* limit, by continuity of $\pi_r$ we have weak* convergence 
$(\pi_r)_*\rho_{n_k} \to (\pi_r)_*\rho$. 
By the claim, it follows  that $(\pi_r)_*\rho = \nu^{\otimes r}$. 

By minimality of the action of $G$ on $X$, for every nonempty open set $U\subset X$ we have $\nu(U) > 0$. Consider $A = \{(K,x_{1},\ldots) \in \CK(X) \times X^{\mathbb{N}} : x_{i} \not \in K \textrm{ for some } i\}$.  Since every $K \in \CK(X)$ is compact. this set is open.  
By construction, $\rho_{0}(A) = 0$ 
and so $\rho_{n}(A)=0$ for every $n \ge 1$, and hence $\rho(A) = 0$ by Portmanteau. Thus $\rho$-almost every tuple $(K,x_{1},\ldots)$ has the property that every $x_{i}$ belongs to $K$.

If $K\cap U = \emptyset$, then every sampled coordinate $x_i\in K$ lies outside of $U$ and so for every $r\geq 1$, 
$$
\mu(\{K: K\cap U = \emptyset\})\leq \rho(x_1, \dots, x_r \in X\setminus U) = \nu(X\setminus U)^r,
$$
using the $r$-fold product marginal. 
As $\nu(U)> 0$, as $r\to\infty$ the last term tends to $0$, and so 
$$\mu(\{K: K\cap U = \emptyset\}) =0.
$$
Letting $(U_m)_{m\in\N}$ be a countable base for $X$, it follows that 
$$
\mu(\{K: K\cap U_m\neq\emptyset \text{ for all } m\in \N\} = 1. $$
But since any compact set meeting every nonempty base element must be $X$, we have that 
$\mu(\{X\}) =1$ and $\mu = \delta_X$.

\end{proof}

\begin{corollary}
Let $G$ be an amenable group acting on the compact metric space $X$ such that $(X,G)$ is minimal, uniquely ergodic, and has countable exceptions for products. Let $(F_{m})_{m=1}^\infty$ be a F\o{}lner sequence for $G$. Then for every compact subset $Y \subset X$, precisely one of the following holds:
\begin{enumerate}
\item\label{item:go-to-zero}
For every closed set $E \subset \CK_{\textrm{uc}}(X)$, we have
$$\limsup_{m \to \infty}\frac{1}{|F_{m}|}\big\vert \{g \in F_{m} : gY \in E\}\big\vert  < 1.$$
\item\label{item:go-to-one}
For every $\varepsilon > 0$ we have
$$\limsup_{m \to \infty}\frac{1}{|F_{m}|}\big\vert\{g \in F_{m} : d(gY,X) < \varepsilon\}\big\vert = 1.$$
\end{enumerate}
\end{corollary}
\begin{proof}
Let $Y \subset X$ be compact and for each $m\geq 1$, set $\mu_{m} = \frac{1}{|F_{m}|}\sum_{g \in F_{m}}\delta_{gY}$. It follows from Theorem~\ref{thm:msjweaklyicrigid} that if $\mu$ is a weak* accumulation point of the sequence $\{\mu_{m}\}$ and $\tilde{\mu}$ is its completion, then $\tilde{\mu}(\CK_{\textrm{uc}}(X)) < 1$ or $\mu = \delta_{X}$. Item~\eqref{item:go-to-one} in the statement is precisely the second case, and so it suffices to show that~\eqref{item:go-to-zero} holds assuming that every weak* accumulation point $\mu$ of $\{\mu_{m}\}_{m\in\N}$ satisfies $\tilde{\mu}(\CK_{\textrm{uc}}(X)) < 1$. Let $E \subset \CK_{\textrm{uc}}(X)$. 

If~\eqref{item:go-to-zero} does not hold, then there is some sequence   $\{m_{k}\}_{k\in\N}$ satisfying
$$\frac{1}{|F_{m_{k}}|}\big\vert\{g \in F_{m_{k}} : gY \in E\}\big\vert \to 1$$
as $k \to \infty$, 
and by passing to a subsequence if necessary we can assume that the measures $\mu_{m_{k}}$ weak* converge to some measure $\mu$.  
The Portmanteau Theorem then implies that $\mu(E) \ge \limsup_{k \to \infty}\mu_{m_{k}}(E) \ge 1$ along the subsequence, contradicting that $\tilde{\mu}(\CK_{\textrm{uc}}(X)) < 1$.
\end{proof}

\subsection{The Chacon system is weakly IC-rigid but not IC-rigid}
\label{sec:Chacon-not-rigid}
We show that the symbolic Chacon system $(X,T)$ is weakly IC-rigid but not IC-rigid.  This system is rank one, and we make use of basic properties of such systems (see for example~\cite{Ferencz1, Ferencz2}).

To define the Chacon system $(X,T)$, set $b_1 = 0010$ and  define $b_{n+1} = b_nb_n1b_n$ for all $n\geq 1$.  Let $X$ be the subshift in $\{0,1\}^{\mathbb{Z}}$ whose language is the same as the language generated by the sequence $(b_n)_{n\in\N}$ and $T$ denote the shift map $\sigma_{2}$. Let $x_{C} \in X$ be a point such that $x_{C}$ has a unique decomposition into concatenations of blocks $b_{n}$ and the spacer $1$  for every $n\geq 1$. 
It follows from the decomposition of $b_{n}$s into $b_{1}$s and $1$s that the gap between consecutive starting positions of the word $0010$ in $x_{C}$ is either $4$ or $5$. In particular, any word of length $8$ contains at least one full copy of the word $0010$. 

Set 
$$
L = \{|b_{n}|-1: n\geq 2\}.
$$

\begin{lemma}\label{lemma:chaconspacing}
No two copies of $0010$ in $x_{C}$ have starting positions separated by any $\ell\in L$. 
\end{lemma}
\begin{proof}
Toward contradiction, suppose $\ell = |b_{n}|-1$, and assume the word $0010$ occurs at entries $p$ and at $p+\ell$ of $x_{C}$.  Considering the unique way to parse the point $x$ into blocks $b_n$ and the spacer $1$, if two copies occur in the same $b_n$, then the starting positions would be at most $|b_{n}|-4$ apart (or they would not both fit), but $\ell = |b_{n}|-1$.  
So the two copies must lie  in different $b_n$ blocks.  Say they start at places $q_1$ and $q_2$ with $q_1 < q_2$. Then $p = q_1+a$ and $p+\ell = q_2+b$, where $a, b$ are starting places of $0010$ in the same $b_n$. 
Since $\ell < |b_{n}|$, these two $b_n$ blocks need to be consecutive: if they were not, then $q_2-q_1\geq 2|b_{n}|$, and it follows that the distance between the starts of the two copies of $0010$ would be strictly greater than $b_{n}-1$, a contradiction.  
But adjacent blocks are either $|b_{n}|$ apart (when there is no spacer  $1$ inserted) or $|b_{n}|+1$ apart (with a spacer inserted), and so 
$\ell = (q_2-q_1)+(b-a)$ must be either $|b_{n}|+(b-a)$ or $|b_{n}|+1+(b-a)$.  Thus $b-a$ is either $-1$ or $-2$, meaning the starts of these are either $1$ or $2$ apart. But $a$ and $b$ were chosen to be starting positions of $0010$ within the same $b_{n}$, and hence $b-a$ cannot be $-1$ or $-2$.
\end{proof}

By the lemma, for all $n\geq 2$, when the word $0010$ starts at some location in $x_{C}$, any other occurrence of $0010$ in $x_{C}$ does not occur $|b_{n}|-1$ entries away, meaning that each $\ell\in L$ is a forbidden difference for occurrences of $0010$ in $x_{C}$.  

Set 
\begin{equation}
\label{def:Y}
    W = \{T^\ell x_{C}: \ell \in L\} \quad  \text{ and } \quad Y = \overline{W}, 
\end{equation}
and note that  $Y\subset X$ is compact and infinite. 
\begin{lemma}\label{lemma:chaconnondensity}
There exists $\delta > 0$ such that $d(T^{j}(Y),X) \ge \delta$ for all $j \in \mathbb{Z}$. 
\end{lemma}
\begin{proof}
It suffices to prove there exists $\delta > 0$ such that $d(T^{j}(W),X) \ge \delta$ for all $j \in \mathbb{Z}$. For this, it is enough to show that for every $j \in \mathbb{Z}$, there exists some $y(j) \in X$ such that $T^{j}(Y) \cap B_{\frac{1}{16}}(y(j)) = \emptyset$. 
Fixing some $j\in \Z$, we show that $y(j) = T^{j}x_{C}$ satisfies this  conclusion. Note that all $z\in B_{\frac{1}{16}}(T^jx_{C})$ agree on the center block of length $9$, meaning on entries in $[-4,\dots, 4]$.  
This means that if $d(T^jx_{C}, T^{j+\ell}x_{C}) < 1/16$ for some $\ell\in L$, then 
$$
(x_{C})_{[j-4, j+4]} = (x_{C})_{[j+\ell-4, j+\ell+4]}.$$
Since this is a window of size 9, there is at least one occurrence of $0010$, and if it starts at the entry $j+a$ with $-4\leq a \leq 1$, then we have  occurrences at $j+a$ and $j+\ell+a$, a contradiction of $\ell\in L$ by Lemma~\ref{lemma:chaconspacing}.  
\end{proof}  

\begin{lemma}
\label{lemma:R-exists}
Given $m\geq 1$, there exists $R(m)$ such that for all $j\in\Z$, there exist $y^{(1)}, \dots, y^{(m)}\in T^j(W)$ such that the center blocks $(y^{(i)})_{[R(m), R(m)]}$ are all distinct.
\end{lemma}

\begin{proof}
Fix $m\geq 1$. For each $i$ write $\ell_{i} = |b_{i}|$. Consider the points
$$
y^{(i)} = T^{j+\ell_i-1}x_{C} \quad\text{ for } i=1, \dots, m
$$
and let $R = R(m) = \ell_m$. 
The center blocks of the $y^{(i)}$ correspond to the blocks in $x$ at locations: 
$$
[j+\ell_i-1-\ell_m, j+\ell_i-1+\ell_m] 
\quad\text{ for } i=1, \dots, m.
$$
If two of these blocks are the same, then setting $p_i = j+\ell_i-1$, we have $1\leq a< b\leq m$ with the center indices $p_a$ and $p_b$ giving rise to the same block. But
$p_b-p_a = \ell_b-\ell_a < \ell_m$, and so we have a block of length $2\ell_m+1$ with two occurrences a distance  $\ell_b-\ell_a$ apart.  Looking in the first block, it must contain the word $b_m$ since it has length $2\ell_m+1$, and so we have two occurrences of $b_m$ that are a distance $\ell_b-\ell_a< \ell_m$ apart. But two occurrences of $b_m$ can only be $\ell_m$ or $\ell_m+1$ apart, and this distance is smaller, 
and so the center blocks $(y_{i})_{[R(m),R(m)]}$ are all distinct.
\end{proof}

\begin{lemma}\label{lemma:chaconboundfromfinites}
For every $r \ge 1$, there exists $\delta(r) > 0$ such that the $T^{\CK}$-orbit of $Y$ in $\CK(X)$ is $\delta(r)$-bounded away from $\CK_{\le r}(X)$, 
meaning that $d(T^{j}(Y),\CK_{\le r}(X)) \ge \delta(r)$ for all $j \in \mathbb{Z}$.
\end{lemma}
\begin{proof}
Fix $r \ge 1$. Applying Lemma~\ref{lemma:R-exists} with $m=r+1$, there exists $R(r+1) \ge 1$ such that for all $j \in \mathbb{Z}$ there exist $r+1$ points $y^{(1)},\ldots,y^{(r+1)} \in T^{j}(W)$ such that $d(y^{(i)},y^{(j)}) \ge 2^{-R(r+1)}$ for all $i \ne j$. 
Then the balls $B(y^{(i)},2^{-R(r+1)-1})$ are pairwise disjoint, and so any set of size $r$ in $X$ can intersect at most $r$ of the balls $B(y^{(i)},2^{-R(r+1)-1})$. Setting $\delta(r) = 2^{-R(r+1)-1}$, it follows that for every $j \in \mathbb{Z}$, we have $d(T^{j}(Y),\CK_{\le r}(X)) \ge \delta(r)$.
\end{proof}

\begin{theorem}
The Chacon system is weakly IC-rigid but not IC-rigid.
\end{theorem}
\begin{proof}
The Chacon system satisfies the hypotheses of Theorem~\ref{thm:msjweaklyicrigid} by combining~\cite{DRS} and~\cite{DK}, and so is weakly IC-rigid. It suffices to show that it is not IC-rigid, meaning we need to show the existence of a nonfinitary nontrivial IRC. 
Taking $x_C$ and $Y$ as in~\eqref{def:Y}, let $\mu$ be a weak* limit of $\frac{1}{n}\sum_{k=0}^{n-1}\delta_{T^{k}Y}$. It follows from Lemma~\ref{lemma:chaconnondensity} that $\mu \ne \delta_{X}$. By Lemma~\ref{lemma:chaconboundfromfinites}, we have 
that $\mu(\CK_{\le r}(X)) = 0$ for every $r \ge 1$, and so $\mu(\CK_{\fin}(X)) = 0$ and $\mu$ is nonfinitary.
\end{proof}

\section{Dilations of the torus}\label{sec:dilations}
\subsection{Dilations of the torus and weak IC-rigidity}
We consider the action of the multiplicative semigroup $\mathbb{N}$ on $\mathbb{T}$ defined by $x \mapsto nx \mod  1$.  We show  that this action is weakly IC-rigid, and then we use this to derive some consequences for dilations of compact subsets of $\mathbb{T}$ satisfying an invariance and entropy assumption.  

Let $T_{n} \colon \mathbb{T} \to \mathbb{T}$ denote the map $T_{n}(x) = nx \mod 1$, and  for $d \in \mathbb{N}$ we also let  $T_{n} \colon \mathbb{T}^{d} \to \mathbb{T}^{d}$ denote the associated diagonal action. Call a subset $S = \{s_{1} < s_{2} < \ldots\} \subset \mathbb{N}$ a \emph{Weyl set} if for all $d\in\N$ and every $\alpha = (\alpha_{1},\ldots,\alpha_{d}) \in \mathbb{T}^{d}$ whose components are independent over the rationals, the sequence $\left(T_{s_{i}} (\alpha)\right)_{i \ge 1}$ is equidistributed, i.e. 
$\frac{1}{N}\sum_{n=1}^{N}\delta_{T_{s_{n}}(\alpha)} \to \lambda$ in the weak* topology (again $\lambda$ denotes Lebesgue measure on $\T^d$). The classic equidistribution theorem of Weyl~\cite{weyl} shows that polynomial subsets are Weyl sets. 

\begin{theorem}\label{thm:propertyPgivesdeltaT}
Suppose $S \subset \mathbb{N}$ is a Weyl set, $\mu$ is a Borel probability measure on $\CK(\mathbb{T})$ whose completion satisfies $\tilde{\mu}(\CK_{\ct}(\mathbb{T})) = 0$, and $\mu$ is invariant under the maps induced by $T_{s}$ for all $s \in S$. Then $\mu = \delta_{\mathbb{T}}$.
\end{theorem}

\begin{proof}
Suppose $\mu$ is such a measure; we assume that $\mu$ is complete by replacing it with its completion. We show that $\mu$-almost every compact set is $\T$. 

Fix a nonempty open interval $I\subset\T$ and set
$$E_I = \{K\in\CK(\T): K\cap I = \emptyset\}.
$$
We show that  that $\mu(E_I) = 0$ for all open intervals $I$.  
Let $K\subset\T$ be an uncountable compact set. 
Then we can choose points $x_1, x_2, \dots\in K$ such that any finite collection of these points is rationally independent in $\T$: first choose $x_1\in K$ irrational, and then inductively choose $x_m$ such that $x_m$ is rationally independent of $x_1, \dots, x_{m-1}$. 

Enumerate $S = \{s_{1}< s_{2}< \ldots\}$. 
 For each fixed $m\in\N$, the sequence $\{T_{s}(x_1, \dots, x_m): s \in S\}$ is equidistributed in $\T^m$ by assumption, and so 
$$
\frac{1}{N}\big\vert \{1\leq n\leq N: T_{s_{n}}(x_i) \notin I \text{ for all } i = 1, \dots, m \}\big\vert\to (1-\lambda(I))^m, 
$$
as $N\to\infty$.    
Set 
$$A_N(K) = \frac{1}{N}\sum_{n=1}^N\one_{\{K : T_{s_{n}}(K)\cap I = \emptyset\}}.
$$
If $T_{s_{n}}(K)\cap I = \emptyset$, then $T_{s_{n}}(x_i) \notin I$ for all $i$ and so 
$$
A_N(K) \leq 
\frac{1}{N}\big\vert \{1\leq n\leq N: T_{s_{n}}(x_i) \notin I \text{ for } i=1, \dots, m\} \big\vert. 
$$
Thus 
$$\limsup_{N\to\infty} A_N(K) \leq (1-\lambda(I))^m.$$
This holds for all $m\geq 1$ and so taking $m\to\infty$, we have that for every uncountable compact set $K$, 
$$
\lim_{N\to\infty} A_N(K) = 0.
$$
Since $\mu$ is $T_{s_{n}}$ invariant for all $n\in\N$, this implies that 
$$
\mu(E_I) = \frac{1}{N}\sum_{n=1}^N\mu(E_I) = 
\frac{1}{N}\sum_{n=1}^N\mu(T_{s_{n}}^{-1}E_I) = 
\int_{\CK(\T)} A_N(K)\,d\mu(K).
$$
The averages $A_N$ are bounded by $1$ and converge to $0$ on the complement of the set of countable compact sets, which by assumption has full $\mu$-measure.  By Dominated Convergence, 
$$
\mu(E_I) = \lim_{N\to\infty}\int_{\CK(\T)} A_N\,d\mu(K) = 0.
$$
Taking a countable basis $\CB$ of open intervals in $\T$, then 
$$
\mu\big(\bigcap_{I\in\CB}\{K:K\cap I\neq\emptyset\}
\big) = 1
$$
and so for $\mu$-almost every compact $K$, we have $K\cap I \neq \emptyset $ for every basis interval $I$, and so $K$ is dense in $\T$.  Since $K$ is compact, we have $K = \T$ and $\mu = \delta_\T$. 
\end{proof}

Theorem~\ref{thm:propertyPgivesdeltaT} immediately implies the following corollaries. 
\begin{corollary}\label{cor:propertyPweaklyicrigid}
If $S \subset \mathbb{N}$ is a semigroup generated by some Weyl set, then the action of $S$ on $\mathbb{T}$ is weakly IC-rigid.
\end{corollary}

\begin{corollary}\label{thm:nontweaklyicrigid}
The action of $\mathbb{N}$ on $\mathbb{T}$ is weakly IC-rigid.
\end{corollary}

\subsection{Dilations of compact subsets}\label{subsec:dilations}

Throughout this section, for ease of exposition 
 we restrict to the $\N$-action by dilation, rather than considering more general subsemigroups of $\mathbb{N}$. However, using Corollary~\ref{cor:propertyPweaklyicrigid} instead of Corollary~\ref{thm:nontweaklyicrigid}, the results immediately generalize to a semigroup action generated by some Weyl set. 

We use weak IC-rigidity of the action of $\mathbb{N}$ on $\mathbb{T}$ to prove some results about dilations of subsets of $\mathbb{T}$. 
Throughout this section, F\o{}lner sequences for $\mathbb{N}$ always mean with respect to the multiplicative, not additive, semigroup structure of $\N$. 
We start with an example showing that for general uncountable compact subsets of $\mathbb{T}$, weak IC-rigidity does not rule out that the images may be condensing onto finite sets. 
\begin{example}\label{example:dilationset}
There exists an uncountable compact subset $Y \subset \mathbb{T}$ such that for the F\o{}lner sequence $F_n = \{p_1^{i_1}p_2^{i_2}\dots p_n^{i_n}: 0 \leq i_j\leq n, 1 \leq j\leq n\}$,   where $\{p_j\}$ is an enumeration of the primes, every accumulation point of the sequences of measures 
$$\mu_{n} = \frac{1}{|F_{n}|}\sum_{m \in F_{n}}\delta_{mY}$$ gives at least weight $\frac{1}{2}$ on $\delta_{\{0\}}$. We sketch the proof. Let $M_{k} = \max F_{k} = \prod_{j=1}^{k}p_{j}^{k-j+1}$ and set $Q(i) = \prod_{j=1}^{m_{i}}p_{j}$, where we choose an increasing sequence of natural numbers $m_{i} \to \infty$ that increases sufficiently quickly such that for all $i$ we have $Q(i+1) > 4 M_{2m_{i}} Q(i)$ and $Q(i+1) > 2^{i+2}M_{2m_{i}}$.  Set $Y = \{\sum_{\ell=1}^{\infty}\frac{\varepsilon_{\ell}}{Q(\ell)} : \varepsilon_{\ell} \in \{0,1\}\}$. Then $Y$ is uncountable and compact. Given $i \ge 1$, if $Q(i)$ divides $n$, then for any $y \in Y$ we have  $d(ny,0) < \frac{1}{2^{i}}$ and hence $d_{H}(nY,0) < \frac{1}{2^{i}}$. But $$\frac{|\{n \in F_{m_{i}} : Q(i) \textrm{ divides } n\}|}{|F_{m_{i}}|} \to \frac{1}{2}$$ as $i \to \infty$.

A similar construction can be used to show that for every $\varepsilon > 0$, there is an uncountable set $Y_{\varepsilon} \subset \mathbb{T}$ such that every accumulation point of $\mu_{n} = \frac{1}{|F_{n}|}\sum_{m \in F_{n}}\delta_{mY}$ gives weight at least $1-\varepsilon$ on $\delta_{\{0\}}$.
\end{example}

Toward proving Theorem~\ref{thm:dilationsintro1}, we need a property to rule out this condensing behavior. 
It turns out that some invariance along with positive entropy suffices for this. We continue to let $T_{n} \colon \mathbb{T} \to \mathbb{T}$ denote the map $T_{n}(x) = nx \mod 1$.

\begin{theorem}\label{thm:invsetsdilations}
Let $Y \subset \mathbb{T}$ be a compact subset invariant under $T_{p}$ for some $p \ge 2$, and suppose $T_{p} \colon Y \to Y$ has positive topological entropy. Then for every F\o{}lner sequence $(F_{m})_{m\in\N}$ of the multiplicative semigroup $\mathbb{N}$ and $\varepsilon > 0$, we have
$$\frac{1}{|F_{m}|}\big\vert\{n \in F_{m} : nY \textrm{ is } \varepsilon\textrm{-dense in } \mathbb{T}\}\big\vert \to 1 $$
 as $m \to \infty$. 
\end{theorem}

\begin{proof}
Let $\mu_{m} = \frac{1}{|F_{m}|}\sum_{n \in F_{m}}\delta_{nY}$. It suffices to show that the only weak* limit of $\mu_{m}$ is $\delta_{\mathbb{T}}$. By Corollary~\ref{thm:nontweaklyicrigid}, the action of $\mathbb{N}$ on $\mathbb{T}$ is weakly IC-rigid, and so it suffices to show that for every weak* limit $\mu$ of $\mu_{m}$, we have $\tilde{\mu}(\CK_{\textrm{ct}}(\mathbb{T})) = 0$.

Fix $\varepsilon > 0$. Since $T_{p} \colon Y \to Y$ has positive topological entropy, by the Variational Principle there exists $\eta \in \mathcal{M}_{T_{p}}(Y)$ such that $h_{\eta}(T_{p}) > 0$. Set $c = h_{\eta}(T_{p})$ and define
$$\CE_{c} = \{Z \in \CK(X) : \textrm{there exists } \eta \in \mathcal{M}_{T_{p}}(Z) \textrm{ satisfying } h_{\eta}(T_{p}) \ge c\}.$$
By assumption, $Y \in \mathcal{E}_{c}$. We claim $nY \in \mathcal{E}_{c}$ for every $n \ge 1$.
Set $\eta_{n} = (T_{n})_{*}(\eta)$. Then $\eta_{n} \in \mathcal{M}_{T_{p}}(nY)$ and $h_{\eta_{n}}(T_{p}) = h_{\eta}(T_{p})$ since $T_{n}$ is $n$-to-one, so the claim follows. Thus $\mu_{m}(\mathcal{E}_{c}) = 1$ for all $m \ge 1$. Note that $\mathcal{E}_{c}$ is closed since the entropy function is upper semicontinuous on the space $\mathcal{M}_{T_{p}}(\mathbb{T})$. Let $\mu$ be a weak* limit of $\mu_{m}$. Then $\mu(\mathcal{E}_{c}) = 1$, and since no countable subset supports a $T_{p}$-invariant measure of positive entropy, we have $\CK_{\textrm{ct}}(\mathbb{T}) \cap \mathcal{E}_{c} = \emptyset$. It follows that $\tilde{\mu}(\CK_{\textrm{ct}}(\mathbb{T})) = 0$, and hence $\mu = \delta_{\mathbb{T}}$.
\end{proof}

\begin{corollary}
\label{cor:syndetic}
Let $Y \subset \mathbb{T}$ be a compact subset invariant under $T_{p}$ for some $p \ge 2$, and suppose $T_{p} \colon Y \to Y$ has positive topological entropy. Then for every $\varepsilon > 0$, the set
$$\{n \in \mathbb{N} : nY \textrm{ is }\varepsilon\textrm{-dense in } \mathbb{T}\}$$
is multiplicatively syndetic.
\end{corollary}

\begin{proof}
This follows by combining Theorem~\ref{thm:invsetsdilations} and the fact that if $S \subset \mathbb{N}$ satisfies $\frac{|S \cap F_{k}|}{|F_{k}|} \to 1$ as $k \to \infty$ for every F\o{}lner sequence for the multiplicative semigroup $\mathbb{N}$, then $S$ is multiplicatively syndetic. Indeed, if $S$ is not multiplicatively syndetic, then for every finite $F\subset\N$, there exists some $t\in\N$ such that 
$FT\subset \mathbb{N} \setminus S$. Let $F_m$ be a F\o{}lner sequence for the multiplicative semigroup $\mathbb{N}$.  For each $m$, pick $t_m\in\N$ such that $F_mt_m\subset \mathbb{N} \setminus S$ and set 
$\Psi_m = F_mt_m$.  Then $(\Psi_m)$ is also a F\o{}lner sequence for $\mathbb{N}$, and by construction lies entirely in $\mathbb{N} \setminus S$.
\end{proof}

\begin{theorem}
\label{th:dilate-revised}
Let $Y \subset \mathbb{T}$ be a compact subset invariant under $T_{p}$ for some $p \ge 2$, and suppose $T_{p} \colon Y \to Y$ has positive topological entropy. Then for every F\o{}lner sequence $(F_{m})_{m\in\N}$ of the multiplicative semigroup $\mathbb{N}$, there exists a set $J$ such that
$$\frac{|J \cap F_{m}|}{|F_{m}|} \to 1 \quad \textrm{ as } \quad  m \to \infty$$
and
$$\lim_{n \to \infty, n \in J}nY = \mathbb{T} \quad \textrm{ in } \quad \CK(\mathbb{T}).$$
\end{theorem}
\begin{proof}
Let $(F_{m})_{m\in\N}$ be a F\o{}lner sequence for the multiplicative action of $\mathbb{N}$. Given $n \in \mathbb{N}$, let $\ell_{n}$ denote the maximal length of an interval in the complement of $nY$ in $\mathbb{T}$. For $r \ge 1$, set $E_{r} = \{n \in \mathbb{N} : \ell_{n} < \frac{1}{r}\}$. By Theorem~\ref{thm:invsetsdilations}, for every $r \ge 1$ we have
$$\frac{|E_{r} \cap F_{m}|}{|F_{m}|} \to 1 \quad \textrm{ as } m \to \infty.$$
Let $I_{r}$ be an increasing sequence of natural numbers such that if $m \ge I_{r}$ then $\frac{|E_{r} \cap F_{m}|}{|F_{m}|} > 1-2^{-r}$.  Define $r(m) = \max\{r : I_{r} \le m\}$. Note that $r(m)$ is increasing and satisfies $r(m) \to \infty$ as $m \to \infty$. Defining $J_{m} = E_{r(m)} \cap F_{m}$, we then have that 
$$\frac{|J_{m}|}{|F_{m}|} = \frac{|E_{r(m)} \cap F_{m}|}{|F_{m}|} > 1 - 2^{-r(m)} \to 1$$ 
as $m \to \infty.$ 
Moreover, since $J_{m} \subset E_{r(m)}$ for every $m$, we have
$$\sup_{n \in J_{m}}\ell_{n} \le \frac{1}{r(m)} \to 0.$$
Setting  $J = \bigcup_{m=1}^{\infty}J_{m}$, it follows that $$\frac{|J \cap F_{m}|}{|F_{m}|} \ge \frac{|J_{m}|}{|F_{m}|} \to 1 
$$ 
as $m \to\infty$.
Given $\varepsilon > 0$,  there exists $M\in\N$ such that for all $m \ge M$ we have $\frac{1}{r(m)} < \varepsilon$, and hence $\sup_{n \in J_{m}} \ell_{n} < \varepsilon$. Since $\bigcup_{m=1}^{M-1}J_{m}$ is finite, there are only finitely many $n \in J$ such that $\ell_{n} \ge \varepsilon$. Altogether it follows that $\lim_{n \to \infty, n \in J}\ell_{n} = 0$ and $\lim_{n \to \infty, n \in J} nY = \mathbb{T}$ in $\CK(\mathbb{T})$.
\end{proof}

\begin{corollary}
Let $Y \subset \mathbb{T}$ be a compact subset invariant under $T_{p}$ for some $p \ge 2$, and suppose $T_{p} \colon Y \to Y$ has positive topological entropy. Let $W \subset \N$ be finite and fix $\varepsilon > 0$. Then for every F\o{}lner sequence $(F_{m})_{m\in\N}$ of the multiplicative semigroup $\N$, the set
$$\frac{1}{|F_{m}|}\lvert \{n \in F_{m} : wnY \textrm{ is }\varepsilon \textrm{-dense in }\mathbb{T} \textrm{ for all } w \in W\}| \to 1 \quad$$
as $m \to \infty$.
\end{corollary}
\begin{proof}
Let $Y$ be such a set, $\varepsilon > 0$, and $(F_{m})_{m \in \N}$ be a F\o{}lner sequence for $\N$. Set $$A_\varepsilon = \{n \in \mathbb{N} : nY \textrm{ is }\varepsilon\textrm{-dense in } \mathbb{T}\}.$$
Let $W \subset \mathbb{N}$ be finite and for $w \in W$ define
\begin{equation}
    \label{eq:R-e-W}
B_{\varepsilon}(w) =  \{n\in\N: wn\in A_\varepsilon(Y)\}.
\end{equation}
Then for each $w \in W$ we have
$$|F_{m} \setminus B_{\varepsilon}(w)| = |wF_{m} \setminus A_{\varepsilon}|.$$
Since
$$wF_{m} \setminus A_{\varepsilon} \subset \left(wF_{m} \setminus F_{m}\right) \cup (F_{m} \setminus A_{\varepsilon}), $$
we have that 
$$\frac{|F_{m} \setminus B_{\varepsilon}(w)|}{|F_{m}|} \le \frac{|wF_{m} \setminus F_{m}|}{|F_{m}|} + \frac{|F_{m} \setminus A_{\varepsilon}|}{|F_{m}|}.$$
As $m \to \infty$, the first term on the right tends to zero since $(F_{m})_{m\in\N}$ is a F\o{}lner sequence, and the second tends to zero by Theorem~\ref{thm:invsetsdilations}. Since $W$ is finite, it follows that
$$\frac{\lvert F_{m} \setminus \bigcup_{w \in W}B_{\varepsilon}(w) \rvert}{|F_{m}|} \to 0 \textrm{ as } m \to \infty$$
which proves the statement.
\end{proof}

\begin{remark}
If one can show that the action of $\mathbb{N}$ on $\mathbb{T}$ is IC-rigid, then one can remove the positive entropy assumption on $Y$ in Theorem~\ref{thm:invsetsdilations}, Corollary~\ref{cor:syndetic}, and Theorem~\ref{th:dilate-revised} and simply assume $Y$ is infinite and invariant under multiplication by $p$ for some $p \ge 2$. 
\end{remark}

\section{Tree structures and deeply transitive actions}\label{sec:treestructuresdeeptransitivity}
\subsection{Tree structures}
Let $X$ be a Cantor set. By a {\em tree structure for $X$} we mean a sequence of clopen partitions $(\mathcal{C}_{i})_{i\in\N} = \{C_{1}^{i},\ldots,C_{\kappa(i)}^{i}\}_{i\in\N}$ of $X$ satisfying  the following:

\begin{enumerate}
\item 
\label{item:tree-one}
For every $i \geq 1$, we have that 
$\CC_{i+1}$ refines $\CC_{i}$,
in the sense that for every $i \ge 1$ and $1 \le j \le \kappa(i)$, $C^{i}_{j}$ is the disjoint union of clopen sets $C^{i+1}_{k}$ for  $k$ belonging to some indexing set $I(i,j)$,
\item 
\label{item:tree-two}
The diameters of the clopens satisfy
$\max \{\diam(C): C \in {\CC}_{i}\} 
\longrightarrow 0$ as  $i\to\infty$.
\item 
\label{item:tree-three}
For all $k\geq 1$ and for all $C \in \CC_k$, 
 $$
 \inf_{n\geq k}\frac{|\{D\in \CC_n: D\subset C\}|}{\kappa(n)} > 0. 
 $$
\end{enumerate}

When the setting is clear, we write $(\CC_{i})_{i\in\N}$ for the tree structure, omitting the enumeration of the clopen sets at each level.  
It is clear that, up to homeomorphism, every Cantor set admits a tree structure.

We note that condition~\eqref{item:tree-two} implies that the set of all $C_{j}^{i}$ form a base for the topology on $X$.  Namely, 
if $x\in X$ and $U$ is an open neighborhood $x$, by setting $\varepsilon = d_H(x, X\setminus U)$, then by condition~\eqref{item:tree-two} we can choose $j\geq 1$ such that the maximum of the diameters is bounded by $\varepsilon$. Since the clopens form a partition of $X$, we can take  the unique clopen containing $x$ and thus obtain a clopen contained in $U$. Furthermore, the set of all finite unions of $C_{j}^{i}$ together with $\emptyset$ form a $\pi$-system, in the sense that the collection is closed under intersections. 

A useful example is the following.
\begin{example}\label{example:treestructureXn}
    Fix $n \ge 2$ and recall (see Section~\ref{sec:subshifts}) that  $X_{n} = \{0,\ldots,n-1\}^{\mathbb{Z}}$ and $X_{n}^{+} = \{0,\ldots,n-1\}^{\mathbb{Z}}$ are the two-sided and one-sided full shifts on $n$ symbols, respectively. 
Setting $\CC_{k} = \{C^{k}_{w} : w \in \CL_{2k+1}(n)\}$ for each $k \ge 1$, where 
$\CL_{2k+1}$ denotes the words  of length $2k+1$ defined by the (symmetric) cylinder sets $[w]$, this 
defines a tree structure on $X_{n}$. 
An analogous tree structure is defined on $X_{n}^{+}$ using cylinder sets of words based at 1.
\end{example}

Given a tree structure $(\CC_{i})_{i\in\N}$ on $X$,  for every $m \ge 2$ there are projection maps
$P_{m} \colon \CC_{m} \to \CC_{m-1}$
defined by $P_{m}(C^{m}_{i}) = C^{m-1}_{j(i)}$ where $C^{m-1}_{j(i)}$ is the unique clopen in $C_{m-1}$ which contains $C^{m}_{i}$.

The sets $\CC_{m}$ are finite, and so endowing them with the discrete topology, we obtain induced maps
$$P_{m}^{\CK} \colon \CK(\CC_{m}) \to \CK(\CC_{m-1}).$$
We note that $\CK(\CC_{m})$ is the power set (omitting the empty set) of $\CC_{m}$, but we use the $\CK$ notation for consistency.

We give a useful result that allows us to work with the space $\varprojlim\{\CK(\CC_{m}),P_{m}^{\CK}\}$, rather than $\CK(X)$.
\begin{theorem}\label{thm:inverselimitpresentationK}
Suppose that $\{\CC_{m}\}_{m\in\N}$ is a tree structure for $X$. Then there is a homeomorphism $F \colon \CK(X) \to \varprojlim\{\CK(C_{m}),P_{m}^{\CK}\}$. 
\end{theorem}

\begin{proof}
Given $Y \in \CK(X)$, let $A_{m}(Y) \in \CK(\CC_{m})$ be the set of $C \in \CC_{m}$ such that $Y \cap C \ne \emptyset$.  We  define $F(Y) = (A_{m}(Y))_{m \in \mathbb{N}}$, and check that this is well defined.  First we claim that $P_{m}^{\CK}(A_{m}(Y)) = A_{m-1}(Y)$ for every $m \ge 2$. Given $P_{m}^{\CK}(C) \in P_{m}^{\CK}(A_{m}(Y))$, choose $y \in Y$ such that $y \in C$. Then $y \in P_{m}(C)$ since $C \subset P_{m}(C)$, and so $P_{m}(C) \in A_{m-1}(Y)$ and $P_{m}^{\CK}(A_{m}(Y)) \subset A_{m-1}(Y)$. For the other inclusion, suppose $C \in A_{m-1}(Y)$ with $y \in Y \cap C$. There exists $D \in \CC_{m}$ such that $D \subset C$ and $y \in D$. Then $D \in A_{m}(Y)$ and $P_{m}(D)=C$, and so $A_{m-1}(Y) \subset P_{m}^{\CK}(A_{m}(Y))$. Thus it follows that $(A_{m}(Y))_{m \in \mathbb{N}} \in \varprojlim\{\CK(\CC_{m}),P_{m}^{\CK}\}$, and we have the opposite inclusion, and $F$ is well defined. 
To check that $F$ is continuous, suppose $m \ge 1$. Let $\delta = \frac{1}{2}\min \{d(C,D) : C \ne D \in \CC_{m}\}$. If $d(Y,Y^{\prime}) < \delta$, then $Y \cap E \ne \emptyset$ if and only if $Y^{\prime} \cap E \ne \emptyset$ for all $E \in \CC_{m}$, and hence $A_{m}(F(Y)) = A_{m}(F(Y^{\prime}))$. 

We next define a map $H \colon \varprojlim\{\CK(\CC_{m}),P_{m}^{\CK}\} \to \CK(X)$ which is an inverse to $F$. Given $(A_{m})_{m \in \mathbb{N}} \in \varprojlim\{\CK(\CC_{m}),P_{m}^{\CK}\}$, set $H(A_{m}) = \bigcap_{m=1}^{\infty}\bigcup_{C \in A_{m}} C$. Note that the image $H(A_{m})$ is compact and nonempty, since it is the intersection of a nested sequence of compact sets. To check continuity of $H$, given $m \ge 1$, say a pair $Y,Y^{\prime}$ are $m$-close if $Y \cap C \ne \emptyset$ if and only if $Y^{\prime} \cap C \ne \emptyset$ for all $C \in \CC_{m}$. Then given $\varepsilon > 0$, there exists $m \ge 1$ such that if $Y,Y^{\prime}$ are $m$-close, then $d(Y,Y^{\prime}) < \varepsilon$, so $H$ is continuous. 

We are left with checking that $F$ and $H$ are inverses. We claim that 
$$Y = H(A_{m}(Y)) = \bigcap_{m=1}^{\infty}\bigcup_{C \in A_{m}(Y)}C.$$ 
Since $Y \subset \bigcup_{C \in A_{m}(Y)}C$ for all $m\geq 1$, it follows that $Y \subset \bigcap_{m=1}^{\infty}\bigcup_{C \in A_{m}(Y)}C$. Conversely, if $y \in \bigcap_{m=1}^{\infty}\bigcup_{C \in A_{m}(Y)}C$, then for every $m\geq 1$ there exists $y_{m} \in Y$ and $C_{m}$ such that $y,y_{m} \in C_{m}$. Since $\diam(C_{m}) \to 0$ as $m \to \infty$, it follows that $y_{m} \to y$. Then since $Y$ is compact, we get that $y \in Y$. This proves the claim, and we have that $H(F(Y)) = Y$. 
For the other direction, consider $(A_{m})_{m\in\N} \in \varprojlim\{\CK(\CC_{m}),P_{m}^{\CK}\}$. We have
$$H((A_{m})_{m\in\N}) = \bigcap_{m=1}^{\infty}\bigcup_{C \in A_{m}}C \subset \bigcup_{C \in A_{m}}C.$$
But also $H(A_{m})$ intersects each $C$ for $C \in A_{m}$, since if $C \in A_{m}$, then for every $k \ge 1$ there exists $C_{k} \in A_{m+k}$ such that $P_{m,k}(C_{k}) = C$, and hence $C_{k} \subset C$. It follows that $F(H((A_{m})_{m\in\N})) = (A_{m})_{m\in\N}$.
\end{proof}

Given a tree structure $(\CC_{i})_{i\in\N}$ on $X$, there is associated to it a natural measure on $X$, the uniform tree measure $\mu_{\CC}$. It is defined inductively by first setting $\mu_{\CC}(C^{1}_{i}) = \frac{1}{\kappa(1)}$ for every $1 \le i \le \kappa(1)$. Then, given $m\geq 1$ we set 
$$\mu_{\CC}(C^{m+1}_{j}) = \frac{\mu(C_{i}^{m})}{|\{k : C^{m+1}_{k} \subset C^{m}_{i}\}|} \quad\text{  for every } C^{m+1}_{j} \subset C^{m}_{i}.
$$
For example, for the tree structure $(\CC_{i})$ on $X_{n}$ defined in Example~\ref{example:treestructureXn}, $\mu_{\CC}$ is the $(\frac{1}{n},\ldots,\frac{1}{n})$-Bernoulli measure on $X_{n}$.

\subsection{Deeply transitive actions}
\label{sec:deeply-extremely}

\begin{definition} 
    Let $G$ be a group acting on $X$ and let $(\CC_{i})_{i \in \mathbb{N}}  = \{C_{1}^{i},\ldots,C_{\kappa(i)}^{i}\}_{i\in\N}$ be a tree structure on $X$. 
    \begin{enumerate}
        \item The action of $G$ is 
    \emph{deeply transitive with respect to the tree structure $(\CC_{i})_{i\in\N}$} if for each $i\geq 1$, $1 \le r \le \kappa(i)$, and all pairs of sets $\{A_1^i, \dots, A_r^i\}, \{B_1^i, \dots, B_r^i\}$ where $A_k^i, B_k^i \in \CC_i$ for $k=1, \dots, r$, there exists $g\in G$ satisfying $\{g(A_1^i), \dots, g(A_r^i)\}$ $= \{B_1^i, \dots, B_r^i\}$.  
    \item 
    The action of $G$ is 
     \emph{extremely transitive with respect to the tree structure $(\CC_{i})_{i\in\N}$} if for each $i\geq 1$, $1 \le r \le \kappa(i)$, and all pairs of distinct tuples of sets $(A_1^i, \dots, A_r^i)$, $(B_1^i, \dots, B_r^i)$ where $A_k^i, B_k^i \in \CC_i$ for $k=1, \dots, r$, there exists $g\in G$ such that  $(g(A_1^i), \dots, g(A_r^i)) = (B_1^i, \dots, B_r^i)$.  
    \end{enumerate}
    We say $G$ acts {\em deeply (respectively, extremely) transitively} on $X$ if there exists a tree structure $(\CC_{i})_{i \in \mathbb{N}}$ on $X$ such that $G$ acts deeply (respectively, extremely) transitively with respect to  $(\CC_{i})_{i \in \mathbb{N}}$.
\end{definition}
In other words, when $G$ acts deeply transitively, at each level $i$ in the tree, $G$ acts transitively on $r$-element subsets of $\CC_i$ for every $r\geq 1$, meaning that it is set-multiply transitive on every finite level. 

If $G$ acts deeply (or extremely) transitively on $X$, then we refer to an associated tree structure for which it acts deeply (respectively, extremely) transitively as a {\em witness}. Clearly if $G$ acts extremely transitively then it acts deeply transitively. We note that, as the definitions are existence statements, if $G \subset G^{\prime}$ is a subgroup and $G$ acts deeply (extremely) transitively on $X$, then so does $G^{\prime}$.

\begin{proposition}\label{prop:deeptransitivecenter}
Suppose $G$ acts faithfully and deeply transitively on a Cantor set $X$. Then $G$ has trivial center.
\end{proposition}
\begin{proof}
Let $(\CC_{i})_{i\in\N} = \{C_{1}^{i},\ldots,C_{\kappa(i)}^{i}\}_{i\in\N}$ be a tree structure witnessing the deeply transitive action. Suppose $g\in G$ with $g \ne \Id$, and choose $x \ne y \in X$ such that $g(x) = y$. 
Then there exist $i$ and $j \ne k$ such that $x \in C^{i}_{j}, y \in C^{i}_{k}$, 
and we can further assume that $i$ is sufficiently large such that $\kappa(i) \ge 3$. Choose $i^{\prime} \ge i$ such that $x \in C^{i^{\prime}}_{p}$ for some $p$, and $g(C^{i^{\prime}}_{p}) \subset C^{i}_{k}$, and $y \in C^{i^{\prime}}_{q}$ for some $q$. Since $\kappa(i) \ge 3$, we can choose some $r$ such that $C^{i^{\prime}}_{r}$ is disjoint from $C^{i}_{k}$. 
By the assumption that the action is deeply transitive, there exists $h \in G$ such that $h(C^{i^{\prime}}_{p}) = C^{i^{\prime}}_{p}$ and $h(C^{i^{\prime}}_{q}) = C^{i^{\prime}}_{r}$. Then $g(x)=y \in C^{i^{\prime}}_{q}$ and so $hg(x) \in C^{i^{\prime}}_{r}$. On the other hand, $h(x) \in C^{i^{\prime}}_{p}$ so $gh(x) \in C^{i}_{k}$ which is disjoint from $C^{i^{\prime}}_{r}$. 
Thus we have that $gh(x) \ne hg(x)$, and since $g$ is an arbitrary non-identity element, the center is trivial. 
\end{proof}

Recall that $(X, G)$ is {\em prime} if it has  no nontrivial factors.

\begin{proposition}\label{prop:deeptransitivityprime}
If $G$ acts deeply transitively on a Cantor set $X$, then $(X,G)$ is minimal and prime.
\end{proposition}
\begin{proof}
To show minimality, let $(\CC_i)_{i\in\N}$ be a tree structure witnessing the deep transitivity. Let $x,y \in X$ and $U$ be an open neighborhood of $y$. By definition, there exist $i,j,k$ such that $y \in C^{i}_{j} \subset U$ and $x \in C^{i}_{k}$. The deep transitivity then gives $g \in G$ such that $g(C^{i}_{k}) = C^{i}_{j}$, and so $g(x) \in U$. This implies the orbit of $G$ is dense, and $(X,G)$ is minimal.

To check that $(X, G)$ is prime, suppose that $(Y, G)$ is a factor with factor map $\pi \colon X \to Y$. Consider $F(\pi) = \{(x,y) \in X^{2} : \pi(x) = \pi(y)\}$.  
Then $F(\pi)$ is invariant under the diagonal action of $G$ on $X^{2}$. 
Suppose that $\pi$ is not injective, and so $F(\pi) \ne \Delta(X)$, and let $(a_{1},a_{2}) \in F(\pi)$ satisfy $a_1 \ne a_2$. Let $(x_{1},x_{2}) \in X^{2}$ and $\varepsilon > 0$. Since $X$ is Cantor, we may choose $(y_{1},y_{2}) \in X^{2}$ such that $y_{1} \ne y_{2}$ and $d(x_{i},y_{i}) < \varepsilon$ for $i=1,2$. Let $(\CC_i)_{i\in\N} = \{C_{i}^{1},\ldots,C^{i}_{\kappa(i)}\}$ be a tree structure witnessing the deep transitivity. We may choose distinct $C^{m}_{k_{1}}, C^{m}_{k_{2}}, C^{m}_{j_{1}}, C^{m}_{j_{2}}$ such that $y_{i} \in C^{m}_{k_{i}}, a_{i} \in C^{m}_{j_{i}}$, and the diameters of all the $C^{m}$ are less than $\varepsilon$. By the assumption of deep transitivity, there exists $g \in G$ such that $\{g(C^{m}_{k_{1}}),g(C^{m}_{k_{2}})\} = \{C^{m}_{j_{1}},C^{m}_{j_{2}}\}$. Since the diameters of the $C^{m}$ are less than $\varepsilon$, this implies either $d(g(a_{i}),y_{i}) < \varepsilon$ for $i=1,2$ or $d(g(a_{i}),y_{(i+1)\!\!\mod 2} ) < \varepsilon$ for $i=1,2$. Since $F(\pi)$ is $G$-invariant, it follows that either $(x_{1},x_{2}) \in F(\pi)$ or $(x_{2},x_{1}) \in F(\pi)$. But $F(\pi)$ is invariant under the map $(u,v) \mapsto (v,u)$, and so in either case, $(x_{1},x_{2}) \in F(\pi)$ and $F(\pi) = X^{2}$. Thus we have that  $\pi(u) = \pi(v)$ for all $u,v \in X$, and so $Y$ is a trivial one point system.
\end{proof}

\subsection{Examples}
\label{sec:ICridigexamples}
We give some examples of deeply transitive actions; in fact, all of the examples are extremely transitive. Given a space $X$, we let $\Homeo(X)$ denote the group of self-homeomorphisms of $X$.

\subsubsection{Stabilized automorphism groups.} Let $n \ge 2$ and consider the full shift on $n$ symbols $\sigma_{n} \colon X_{n} \to X_{n}$. The automorphism group $\aut(\sigma_{n}) = \aut(X_n, \sigma_{n})$ of the system $(X_{n},\sigma_{n})$ is defined to be the group 
$$\aut(\sigma_{n}) = \{\varphi \in \Homeo(X_{n}) : \varphi \sigma_{n} = \sigma_{n} \varphi\}$$
where the group operation is composition (see~\cite{hedlund, BLR}  for background on this well studied group). The stabilized automorphism group of $(X_{n},\sigma_{n})$ is defined to be the group $$\autinf(\sigma_{n}) = \{\varphi\in\Homeo(X_n): \varphi \sigma_{n}^{k} = \sigma_{n}^{k} \varphi \textrm{ for some } k\geq 1\}$$
(see~\cite{HKS, S22} for background on this group). 
It is clear that $\aut(\sigma_{n}) \subset \autinf(\sigma_{n})$. The group $\autinf(\sigma_{n})$ acts extremely transitively on $X_{n}$ for the tree structure defined in Example~\ref{example:treestructureXn}. We show this is a consequence of the extreme transitivity of a certain locally finite subgroup. 

For each $k \ge 1$, we define an action of  the group $\sym(n^{2k+1})$ on $X_{n}$. Enumerating the words of length $2k+1$ over $\{0,\ldots,n-1\}$ as $\mathcal{L}_{2k+1}(X_{n}) = \{w_{1},\ldots,w_{n^{2k+1}}\}$, we  have $\tau\in \sym(n^{2k+1})$ act on $\mathcal{L}_{2k+1}(X_{n})$ by $\tau(w_{i}) = w_{\tau(i)}$.
Then given $x \in X_{n}$, consider $x$ as the concatenation of words of length $2k+1$
$$x = \ldots u_{-2}u_{-1}u_{0}u_{1}u_{2} \ldots$$
where $u_{i} = x_{[i(2k+1)-k,i(2k+1)+k]}$, and for $\tau\in\sym(n^{2k+1})$, define
$$\tau(x) = \ldots \tau(u_{-2}) \tau(u_{-1}) \tau(u_{0}) \tau(u_{1}) \tau(u_{2}) \ldots$$
and so $\tau(x)_{[i(2k+1)-k,i(2k+1)+k]} = \tau(u_{i})$. For each $k\geq 1$ we identify $\sym(n^{2k+1})$ with a subgroup $S(k)$ of $\Homeo(X_{n})$ via these actions, and then define
$$\CS_{n} = \bigcup_{k=1}^{\infty}S(k)$$
where the union is taken in $\Homeo(X_{n})$. Since each $S(k)$ is contained in $\aut(\sigma_{n}^{2k+1})$, it follows that $\CS_{n}$ is a subgroup of $\autinf(\sigma_{n})$. These subgroups play an important role in~\cite{HKS, S22, salo}, and they are simple for every $n \ge 2$. 

It is straightforward to check that the action of $\CS_{n}$ on $X_{n}$ is extremely transitive for the tree structure defined in Example~\ref{example:treestructureXn}. It follows that $\autinf(\sigma_{n})$ also acts extremely transitively on $X_{n}$ for the same tree structure.

Note that since $\aut^{\infty}(\sigma_{n})$ is not simple (it factors onto a nontrivial free abelian group, see~\cite{HKS}) and acts faithfully, a group $G$ with a faithful deeply transitive action need not be simple.

\subsubsection{AF full groups.}\label{subsubsection:affullgroups} Another collection of locally finite examples is obtained as follows. Again fix $n \ge 2$, and for each $k \ge 1$ enumerate the words $\mathcal{L}_{k}(X_{n}^{+}) = \{w_{1},\ldots,w_{n^{k}}\}$. We then have $\sym(n^{k})$ act on $\mathcal{L}_{k}(X_{n}^{+})$ by $\tau(w_{i}) = w_{\tau(i)}$, and act on $X_{n}^{+}$ as follows: for $\tau \in \sym(n^{k})$ and $x \in X_{n}^{+}$ we define
$$\tau(x) = \tau\left(x_{1}\ldots x_{k}\right)x_{k+1}\ldots.$$
This identifies each $\sym(n^{k})$ with a subgroup $F(k)$ of $\Homeo(X_{n}^{+})$, and we define
$$\mathcal{F}_{n} = \bigcup_{k=1}^{\infty}F(k)$$
where the union is taken in $\Homeo(X_{n}^{+})$. Then for the tree structure defined by setting $\mathcal{C}_{i}$ to be the collection of cylinder sets of words of length $i$ based at $1$ (see Example~\ref{example:treestructureXn}), the group $\mathcal{F}_{n}$ acts extremely transitively with respect to $\mathcal{C}_{i}$.

The group $\mathcal{F}_{n}$ is the AF full group of the $n$-adic odometer, and much has been studied concerning their invariant random subgroups and characters, see for instance~\cite{Dudko,DudkoMedynetsAF,DudkoMedynetsAF2}).

\subsubsection{Thompson's group $V$.} Thompson's group $V$ acts on the one-sided full shift $X_{2}^{+} = \{0,1\}^{\mathbb{N}}$ and this action is extremely transitive for the standard tree structure on $X_{2}^{+}$ (see~\cite{higman, CFP} for background on $V$). More generally, for each $d \ge 2$ the Higman-Thompson group $V_{d,1}$ acts extremely transitively on $X_{d}^{+}$ with respect to the standard tree structure (this can be deduced from~\cite{CFP} and is made explicit in~\cite{BrinBook}) Note that these are examples of finitely presented simple groups acting extremely transitively.

\begin{remark}
It is possible for a group $G$ to act deeply but not extremely transitively with respect to a given tree structure $\mathcal{C}_{i}$. Indeed, consider the group $\mathcal{F}_{n}$ defined in Example~\ref{subsubsection:affullgroups} and consider the subgroup $\mathcal{A}_{n} = \bigcup_{k=1}^{\infty}\textrm{Alt}(n^{k}) \subset \mathcal{F}_{n}$. Consider the tree structure $\mathcal{C}_{i}$ on $X_{n}^{+}$ defined in Example~\ref{example:treestructureXn}. Since the alternating group $\textrm{Alt}(m)$ acts $k$-multiply-set-transitively on $\{1,\ldots,m\}$ for each $1 \le k \le m$ but only $k$-transitively for $1 \le k \le m-2$, it follows that $\mathcal{A}_{n}$ acts deeply but not extremely transitively with respect to the tree structure $\mathcal{C}_{i}$ on $X_{n}^{+}$. 
\end{remark}

\subsection{Deeply transitive actions and IC-rigidity}
The goal of this section is to prove that deeply transitive actions are IC-rigid. 
\begin{theorem}
\label{th:FI}
Every deeply transitive action on a Cantor set $X$ is IC-rigid.
\end{theorem}
In other words, the only nonfinitary IRC for a deeply transitive action is $\delta_{X}$.

The remainder of this section is devoted to the proof of Theorem~\ref{th:FI}. We fix, for the remainder of the section, a space $X$, a group $G$ acting on $X$, and a tree structure $(\CC_{i})_{i\in\N}$ on $X$.
For every $m \ge k \ge 1$, we define
$$\CF_{m}(k) = (P^{\CK}_{m,k})^{-1}(\CC_{k}) \subset \CK(\CC_{m}),$$
where we are considering $\CC_{k} \in \CK(\CC_{k})$. Thus $\CF_{m}(k)$ is the collection of all sets of partition elements from  level $m$ which intersect every element of $\CC_{k}$. In other words,
$$\CF_{m}(k) = \{W \subset \CC_{m}: \textrm{ for every } C \in \CC_{k} \textrm { there exists  } D \in W  \textrm{ such that } D \subset C\}.$$

Let $\CF_{m,r}(k) = F_{m}(k) \cap \CK_{r}(\CC_{m})$, meaning this is the collection of sets of size $r$ of partition elements from level $m$ which intersect every element of $\CC_{k}$.

The main technical lemma is the following.
\begin{lemma}
\label{lemma:good-range}
Let $k \ge 1$ and let $\varepsilon > 0$. There exists $R > 0$, which only depends on $k$ and $\varepsilon$, such that for all $i\in\N$ sufficiently large and all $R \le r \le \kappa(k+i)$, we have
$$\frac{|\CF_{k+i,r}(k)|}{|\CK_{r}(\CC_{k+i})|} > 1 - \varepsilon.$$
\end{lemma}

\begin{proof}
Let $C \in \CC_{k}$. Define $\Ext(C,i)$ to be the set of $D$ in $\CC_{k+i}$ which are contained in $C$ and set $e(C,i) = |\Ext(C,i)|$. 
Then the number of subsets of size $r$ of $\CC_{k+i}$ containing no element from $\Ext(C,i)$ is $\binom{\kappa(k+i)-e(C,i)}{r}$. Thus the probability of selecting a subset of size $r$ from $\CC_{k+i}$ which contains no element from $\Ext(C,i)$ is 
$$\frac{\binom{\kappa(k+i)-e(C,i)}{r}}{\binom{\kappa(k+i)}{r}} 
= \prod_{p=0}^{r-1}\frac{\kappa(k+i)-e(C,i)-p}{\kappa(k+i)-p}.$$ 
Set $f(C,i) = \frac{\kappa(k+i)}{e(C,i)}$. Then for each integer $p\geq 0$, we have
$$\frac{\kappa(k+i)-e(C,i)-p}{\kappa(k+i)-p} \le \frac{\kappa(k+i)-e(C,i)}{\kappa(k+i)}=1 - \frac{e(C,i)}{\kappa(k+i)} = 1 - \frac{1}{f(C,i)}.$$
Thus
$$\prod_{p=0}^{r-1}\frac{(\kappa(k+i)-e(C,i)-p}{\kappa(k+i)-p} \le \left(1- \frac{1}{f(C,i)}\right)^{r} \le e^{-\frac{r}{f(C,i)}}.$$
It follows that the probability that a randomly chosen subset of $\CC_{k+2i}$ of size $r$ does not belong to $\CF_{k+i,r}(k)$ is bounded above by
$$\sum_{C \in \CC_{k}}e^{\frac{-r}{f(C,i)}}.$$

By property~\eqref{item:tree-three} in the definition of a tree structure, for each $C \in \CC_{k}$ there exists $L(C) > 0$ such that $\frac{e(C,i)}{\kappa(k+i)} \ge L(C)$. Setting $L_{k} = \min\{L(C) : C \in \CC_{k}\}$, we have that 
$$f(C,i) = \frac{\kappa(k+i)}{e(C,i)} \le \frac{1}{L_{k}} \quad \text{ and } \quad \frac{-r}{f(C,i)} \le -r \cdot L_{k}$$
for all $C \in \CC_{k}$. Thus
$$\sum_{C \in \CC_{k}}e^{\frac{-r}{f(C,i)}} \le \kappa(k)e^{-rL_{k}}$$
and  it follows that
$$1 - \frac{|\CF_{k+i,r}(k)|}{|\CK_{r}(\CC_{k+i})|} \le \kappa(k)e^{-rL_{k}}.$$
Since $k$ is fixed and both $\kappa(k)$ and $L_{k}$ only depend on $k$, the right hand side tends to zero as $r \to \infty$. Thus there exists $R > 0$, which only depends on $k$ and $\varepsilon$, such that $\kappa(k)e^{-rL_{k}} \le \varepsilon$ for $r \ge R$, and hence
$$\frac{|\CF_{k+i,r}(k)|}{|\CK_{r}(\CC_{k+i})|} \ge 1-\varepsilon$$
for such $r$.
\end{proof}

For each $m\geq 1$,  let $P^{\infty}_{m} \colon \varprojlim\{\CK(\CC_{i}),P^{\CK}_{i}\} \to \CK(\CC_{m})$ denote the projection to the $m\textsuperscript{th}$ coordinate. Recall that by Theorem~\ref{thm:inverselimitpresentationK},  we have a homeomorphism $F \colon \CK(X) \to \varprojlim\{\CK(\CC_{m}),P_{m}^{\CK}\}$.

\begin{lemma}\label{lemma:decreasingmass}
Let $\mu \in \CM(\CK(X))$, let $\nu = F_{*}(\mu) \in \CM(\varprojlim\{\CK(\CC_{m}),P^{\CK}_{m}\}))$, and for each $m\geq 1$, let $\nu_{m} = (P^{\infty}_{m})_{*}(\nu)$ be the projection of $\nu$ to $\CM(\CK(\CC_{m}))$. If $\mu(\mathcal{K}_{\fin}(X)) = 0$, then for every $r \ge 1$ we have $\nu_{m}(\CK_{\le r}(\CC_{m})) \to 0$ as $m \to \infty$.
\end{lemma}
\begin{proof}
Let $r \ge 1$. Note that for every $m \ge 1$, we have
$$\nu_{m}(\CK_{\le r}(\CC_{m})) = (P_{m}^{\infty})_{*}(\nu)(\CK_{\le r}(\CC_{m})) = \nu((P_{m}^{\infty})^{-1}(\CK_{\le r}(\CC_{m}))).$$
Moreover,
$$F(\CK_{\le r}(X_{n})) = \bigcap_{m=1}^{\infty}(P^{\infty}_{m})^{-1}(\CK_{\le r}(\CC_{m}))$$
and $\mu(\CK_{\le r}(X)) = 0$, so
$$\nu\left(\bigcap_{m=1}^{\infty}(P^{\infty}_{m})^{-1}(\CK_{\le r}(\CC_{m}))\right) = 0.$$
Since the sets $(P^{\infty}_{m})^{-1}(\CK_{\le r}(\CC_{m}))$ form a nested sequence, we have $$\nu_{m}(\CK_{\le r}(\CC_{m})) = \nu((P^{\infty}_{m})^{-1}(\CK_{\le r}(\CC_{m}))) \to 0$$
as $m \to \infty$.
\end{proof}

\begin{lemma}\label{lemma:eventualfullweight}
Suppose $\mu \in \CM(\CK(X))$, let $\nu = F_{*}(\mu) \in \CM(\varprojlim\{\CK(\CC_{m}),P^{\CK}_{m}\}))$, and for each $m\geq 1$,  let  $\nu_{m} = (P_{m}^{\infty})_{*}(\nu)$ be the $m\textsuperscript{th}$ projection of $\nu$. If $\nu_{m}(\CC_{m}) = 1$ for all sufficiently large  integers $m$, then $\mu = \delta_{X}$.
\end{lemma}
\begin{proof}
For each $m$, let $B_{m} = (P^{\infty}_{m})^{-1}(\CC_{m}) \subset \varprojlim\{\CK(\CC_{m}),P^{\CK}_{m}\}$. Then
$$\nu(B_{m}) = \nu((P^{\infty}_{m})^{-1}(\CC_{m})) = \nu_{m}(\CC_{m}) = 1$$
for $m$ sufficiently large. Since
$$\bigcap_{m=1}^{\infty}B_{m} = (\ldots,\CC_{m},\CC_{m+1},\ldots) \in \varprojlim\{\CK(\CC_{m}),P^{\CK}_{m}\},$$
it follows that $\nu(\{\ldots,\CC_{m},\CC_{m+1},\ldots\}) = 1$. But $F(X) = (\ldots,\CC_{m},\CC_{m+1},\ldots)$, and so $\mu(\{X\}) = 1$ and $\mu = \delta_{X}$.
\end{proof}

\begin{lemma}\label{lemma:equalsizeequalmass}
Suppose $G$ acts deeply transitively on a Cantor set $X$ and $\mu$ is an IRC for the action. Let $\nu = F_{*}(\mu) \in \CM(\varprojlim\{\CK(\CC_{m}),P^{\CK}_{m}\})$ and $\nu_{m} = (P^{\infty}_{m})_{*}(\nu)$ be the projections of $\nu$. For every $m \ge 1$, if $A,B \in \CK(\CC_{m})$ satisfy $|A|=|B|$, then $\nu_{m}(A) = \nu_{m}(B)$.
\end{lemma}
\begin{proof}
Since $F \colon \CK(X) \to  \varprojlim\{\CK(\CC_{m}),P^{\CK}_{m}\}$ is a homeomorphism, we can push the action of $G$ on $\CK(X)$ to an action of $G$ on  $\varprojlim\{\CK(\CC_{m}),P^{\CK}_{m}\}$, and $\nu$ is invariant under this action. Suppose that $A,B \in \CK(\CC_{m})$ satisfy $|A|=|B|$, and write $A = \{C^{m}_{1},\ldots,C^{m}_{r}\}, B = \{D^{m}_{1},\ldots,D^{m}_{r}\}$. Since $\nu_{m} = (P^{\infty}_{m})_{*}(\nu)$, it suffices to show that $\nu((P^{\infty}_{m})^{-1}(A)) = \nu((P^{\infty}_{m})^{-1}(B))$.

It can be seen from the proof of 
Theorem~\ref{thm:inverselimitpresentationK} that $F^{-1}((P_{m}^{\infty})^{-1}(A)) = \{Y \in \CK(X) : Y \cap C^{m}_{i} \ne \emptyset\} \textrm{ if and only if } C^{m}_{i} \in A\}$. By the deeply transitive assumption, there exists $g \in G$ such that $g(\{C^{m}_{1},\ldots,C^{m}_{r}\}) = \{D^{m}_{1},\ldots,D^{m}_{r}\}$. Then $Y \in F^{-1}((P_{m}^{\infty})^{-1}(A))$ if and only if $g(Y) \in F^{-1}((P_{m}^{\infty})^{-1}(B))$, and so $g(F^{-1}((P_{m}^{\infty})^{-1}(A))) = F^{-1}((P_{m}^{\infty})^{-1}(B))$. It follows that, with respect to the $G$-action on $\varprojlim\{\CK(\CC_{m}),P^{\CK}_{m}\})$, $g$ takes $(P^{\infty}_{m})^{-1}(A)$ onto $(P^{\infty}_{m})^{-1}(B)$. Since $\mu$ is $G$-invariant, $\nu$ is $G$-invariant, so $\nu((P^{\infty}_{m})^{-1}(A)) = \nu((P^{\infty}_{m})^{-1}(B))$ as desired.
\end{proof}

We use these lemmas to complete the proof of Theorem~\ref{th:FI}. 
\begin{proof}[Proof of Theorem~\ref{th:FI}]
Let $\mu$ be an IRC for the action and  let $\mu = \nu_{F} + \nu_{I}$ be its  decomposition into the finitary and nonfinitary parts. Let $\nu = F_{*}(\nu_{I})$ and $(\nu_{m})_{m \ge 1} = (P^{\infty}_{m})_{*}(\nu)$ be the projections of $\nu$. By Lemma~\ref{lemma:eventualfullweight}, it suffices to show that $\nu_{m}(\CC_{m}) = 1$ for all $m$ sufficiently large.

Let $\varepsilon > 0$ and let $k \ge 1$. By Lemma~\ref{lemma:good-range}, we may choose $R=R(k,\varepsilon) > 0$ such that
$$\frac{|\CF_{k+i,r}(k)|}{|\CK_{r}(\CC_{k+i})|} \ge 1 - \varepsilon$$
for all $i$ sufficiently large and $R \le r$.

Since $\mu(\CK_{\fin}(X)) = 0$, by Lemma~\ref{lemma:decreasingmass} there exists $J$ such that $\nu_{k+j}(\CK_{\le R}(\mathcal{C}_{k+j})) < \varepsilon$ for all $j \ge J$. 
For each $m\in\N$, set $\CK_{> R}(\CC_{m}) = \CK(\CC_{m}) \setminus \CK_{\le R}(\CC_{m})$. 
Then for $j \ge J$, we have
\begin{equation}\label{eqn:totalsumsofsetsofsizer}
\nu_{k+j}(\CK_{> R}(\CC_{k+j})) \ge 1-\varepsilon.
\end{equation}
Suppose $A,B \in \CK(\CC_{m})$ and $|A| = |B|$. Lemma~\ref{lemma:equalsizeequalmass} then implies $\nu_{m}(A) = \nu_{m}(B)$. 
For each $r,j$, 
 define 
$$t_{k+j}(r) = \nu_{k+j}(\CK_{r}(\CC_{k+j})) = \sum_{B \in \CK_{r}(\CC_{k+j})}\nu_{k+j}(B).$$
Then for every $B \in \CK_{r}(\CC_{k+j})$, we have that
$$\nu_{k+j}(B) = \frac{t_{k+j}(r)}{|\CK_{r}(\CC_{k+j})|}.$$
Defining
$$g_{k+j}(r) = \sum_{B \in \mathcal{F}_{k+j,r}(k)}\nu_{k+j}(B),$$
it follows that 
$$g_{k+j}(r) = \frac{|\CF_{k+j,r}(k)| \cdot t_{k+j}(r)}{|\CK_{r}(\CC_{k+j})|} \ge (1-\varepsilon) t_{k+j}(r)$$
and
$$\sum_{R < r \le \kappa(k+j)} g_{k+j}(r) \ge \sum_{R < r \le \kappa(k+j)}(1-\varepsilon)t_{k+j}(r).$$
By~\eqref{eqn:totalsumsofsetsofsizer}, we have
$$\sum_{R < r \le \kappa(k+j)}t_{k+j}(r) \ge 1 - \varepsilon$$
so 
$$\sum_{R < r \le \kappa(k+j)}g_{k+j}(r) \ge (1-\varepsilon)^{2}.$$
Since $(P^{\mathcal{K}}_{k+j,k})_{*}(\nu_{k+j}) = \nu_{k}$, this implies that 
$$\nu_{k}(\CC_{k}) \ge (1-\varepsilon)^{2}.$$
Since $\varepsilon$ and $k$ are arbitrary, it follows that $\nu_{k}(\CC_{k}) = 1$ for all $k \ge 1$.
\end{proof}

\section{Finitary IRCs of deeply transitive actions}\label{sec:finitaryircs}
\subsection{Invariant measures for deeply transitive actions}
We now turn to classifying all finitary IRCs for deeply transitive actions. Together with the results of the previous section, this gives us a complete classification of all IRCs for deeply transitive actions.

We also classify the self-joinings for extremely transitive actions. The two goals are closely connected, at least in the amenable case. Indeed, if $G$ is amenable and $(X,G)$ is uniquely ergodic with unique $G$-invariant probability measure $\mu$, then every probability measure $\nu$ on $\prod_{i=1}^{k}X$ invariant under the diagonal action of $G$ is automatically a $k$-fold self-joining, since each component is uniquely ergodic. Then, as noted in Proposition~\ref{prop:purelyfinitaryIRCs}, it follows that all finitary IRCs arise as pushforwards under $\rho_{k}$ of self-joinings of $(X,G,\mu)$. In the case that the action is IC-rigid, it follows that all of the IRCs are determined by the self-joinings of $(X,G,\mu)$.

\begin{theorem}\label{thm:deeptransitivityuniqueergodicity}
Suppose $G$ acts deeply transitively on a Cantor set $X$. If $\mu$ is any $G$-invariant Borel probability measure on $X$, then $\mu = \mu_{C}$ for any tree structure $(C_{i})_{i \in \N}$ witnessing the deep transitivity. In particular, if $G$ acts deeply transitively on $X$, then it admits at most one $G$-invariant Borel probability measure, and that measure is necessarily nonatomic.
\end{theorem}
\begin{proof}
Suppose $G$ acts deeply transitively on $X$ and is witnessed by the tree structure $(\CC_{i})_{i\in\N} = \{C_{1}^{i},\ldots,C^{i}_{\kappa(i)}\}$. Suppose $\mu$ is a $G$-invariant Borel probability measure for the action. Let $i \ge 1$. Then for every $1 \le j,k \le \kappa(i)$, there exists $g \in G$ such that $g(C^{i}_{j}) = C^{i}_{k}$, so $\mu(C^{i}_{j}) = \mu(C^{i}_{k})$. Since $1 = \sum_{j=1}^{\kappa(i)}\mu(C^{i}_{j})$, it follows that $\mu(C^{i}_{j}) = \frac{1}{\kappa(i)}$ for every $1 \le j \le \kappa(i)$. 
Since the collection of all $C^{i}_{j}$ form a $\pi$-system, a basis for the topology, and $\mu = \mu_{\CC}$ on this basis, by Dynkin's $\pi$-$\lambda$ Theorem, the two measures agree.

That such a measure $\mu$ is nonatomic follows from the fact that the system $(X,G)$ is minimal by Proposition~\ref{prop:deeptransitivityprime}, and the fact that $X$ is infinite.
\end{proof}

\begin{proposition}
\label{prop:lift}
Suppose $(X,G)$ and $(Y,G)$ are systems and $\pi \colon X \to Y$ is a factor map. Suppose further that there exists $N\geq 1$ such that $|\pi^{-1}(y)| \le N$ for all $y \in Y$. If $\mu \in \CM_{G}(Y)$, then there exists $\nu \in \CM_{G}(X)$ such that $\pi_{*}(\nu) = \mu$. In particular, $\CM_{G}(X) \ne \emptyset$ if and only if $\CM_{G}(Y) \ne \emptyset$.
\end{proposition}
\begin{proof}
Let $\mu$ be a $G$-invariant Borel probability measure on $Y$. Given $y \in Y$, let $\eta_{y} = \frac{1}{|\pi^{-1}(y)|}\sum_{x \in \pi^{-1}(y)}\delta_{x}$ be the counting measure on the fiber $\pi^{-1}(y)$.  Note the map $y \mapsto \eta_{y}$ is Borel as a map $Y \to \CM_G(X)$. 
Then the measure $\nu$ on $X$ defined by $\nu(E) = \int_{Y}\eta_{y}(E)\, d\mu$ is a Borel probability measure on $X$ which projects to $\mu$. Moreover, $\nu$ is $G$-invariant, since for $g \in G$ we have 
\begin{multline*}
    g_{*}(\nu)(E) = \nu(g^{-1}(E)) = \int_{Y}\eta_{y}(g^{-1}(E))\, d \mu \\ = \int_{Y} g_{*}(\eta_{y})(E) \, d \mu = \int_{Y} \eta_{g(y)}(E) \, d \mu = \nu(E). \quad \qedhere
    \end{multline*}
\end{proof}

We apply this in our setting.

\begin{corollary}
\label{co:lift}
For a system $(X,G)$ and $k \ge 1$, the following are equivalent:
\begin{enumerate}
\item
\label{item:MG1}
$\CM_{G}(X) \ne \emptyset$,
 \item
 \label{item:MG2}
 $\CM_{G}(X^{k}) \ne \emptyset$,
\item
\label{item:MG3}
 $\CM_{G}(\CK_{\le k}(X)) \ne \emptyset$.

\end{enumerate}
\end{corollary}
\begin{proof}
Let $k \ge 1$. If $\mu \in \CM_{G}(X)$, then $(\rho_{k})_{*}(\mu^{\otimes k}) \in \CM_{G}(\CK_{\le k}(X))$, and so~\eqref{item:MG1} implies~\eqref{item:MG3}. Since the map $\rho_{k} \colon X^{k} \to \CK_{\le k}(X)$ is a factor map whose fibers are bounded in size by $k!$, it follows from from Proposition~\ref{prop:lift} that~\eqref{item:MG3} implies~\eqref{item:MG2}. Lastly, since $(X^{k},G)$ factors onto $(X,G)$, if $\CM_{G}(X^{k}) \ne \emptyset$ then $\CM_{G}(X) \ne \emptyset$, and so~\eqref{item:MG2} implies~\eqref{item:MG1}.
\end{proof}

\subsection{Finitary IRCs for deeply transitive actions}

For an invariant measure $\nu$ for the action of $G$ on $\CK_{\le s}(X)$, we define
$$\delta(\nu) = \min\{m \le s: \nu(\CK_{\le m}(X)) > 0 \textrm{ and } \nu(\CK_{< m}(X)) = 0\}.$$
This quantity always exists since $\nu(\CK_{\le s}(X)) > 0$. By definition, we always have $\nu(\CK_{\le \delta(\nu)}(X)) > 0$. 
If $\nu$ is ergodic, then ${\delta(\nu) = \min\{m : \nu(\CK_{\le m}(X)) =1\}}$ since $\CK_{\le m}(X)$ is $G$-invariant for every $m\geq 1$.

For the remainder of the section, we assume $G$ acts deeply transitively with respect to a tree structure $(\mathcal{C}_{i})_{i \in \mathbb{N}} = \{C^{i}_{1},\ldots,C^{i}_{\kappa(i)}\}_{i \in \mathbb{N}}$ on a Cantor set $X$.

For $k \ge 1$, let $\CK_{k}(X) = \{Y \in \CK(X) : |Y|=k\}$. Given $i \ge 1$ and $u \in \mathbb{N}^{k}$, we define
$$U^{i}_{k}(u) = \{\{x_{1},\ldots,x_{k}\} \in \CK_{k}(X) : x_{j} \in C^{i}_{u_{j}}\}.$$
At times we also denote this by
$$U_{k}^{i}(u) = \left[C_{u_{1}}^{i},\ldots,C_{u_{k}}^{1}\right].$$
Such sets $U^{i}_{k}(u)$ form a basis for the topology on $\CK_{k}(X)$, and also form a $\pi$-system.
It follows that $\nu$ is determined by its values on such sets $U^{i}_{k}(u)$ for $i \ge 1, u \in \mathbb{N}^{k}$.

\begin{lemma}\label{lemma:knowonevalueknowbasis}
If $\mu, \nu$ are $G$-invariant probability measures on $\CK_{\le k}(X)$ and  $\delta(\mu) = \delta(\nu) = k$, then $\mu = \nu$.
\end{lemma}
\begin{proof}
Considering the measure $\nu$, by assumption we have that $\nu(\CK_{\le k}(X)) > 0$ and $\nu(\CK_{< k}(X)) = 0$, so there exists $u \in \mathbb{N}^{k}$ such that $\nu(U^{1}_{k}(u)) > 0$. Suppose $U^{1}_{k}(u) = \left[C^{1}_{j_{1}},\ldots,C^{1}_{j_{k}}\right]$ and that $U^{1}_{k}(u^{\prime}) = \left[D^{1}_{j_{1}},\ldots,D^{1}_{j_{k}}\right]$ is another such set for some $u^{\prime} \in \mathbb{N}^{k}$. By the deep transitivity of the action, there exists $g \in G$ such that $g(\{C^{i}_{j_{1}},\ldots,C^{i}_{j_{k}}\}) = \{D^{i}_{j_{1}},\ldots,D^{i}_{j_{k}}\}$ and hence $g(U^{1}_{k}(u)) = U^{1}_{k}(u^{\prime})$. Since $\nu$ is $G$-invariant, it follows that
$$\nu(U^{1}_{k}(u)) = \nu(U^{1}_{k}(u^{\prime})).$$

Now let $m \ge 1$ and consider some $U^{m}_{k}(v) = \left[C^{m}_{v_{1}},\ldots,C^{m}_{v_{k}}\right]$ for some $v \in \mathbb{N}^{k}$. For each $1 \le \ell \le k$, using the tree structure can write each $C^{m}_{v_{l}}$ as a disjoint union 
$$C^{m}_{v_{\ell}} = \bigcup_{r \in I(m,v_{l})}C^{m+1}_{r}.$$

We claim $U^{m}_{k}(v)$ is a disjoint union of sets of the form $U^{m+1}_{k}(w)$ for various $w \in \mathbb{N}^{k}$. Indeed, let $J = \prod_{i=1}^{k}I(m,v_{i})$. For $1 \le j \le k$ and $s \in I(m,v_{j})$ we define the set $L(j,s) = \{w \in J : w_{j} = s\}$. Then define
$$R_{j} = \bigcup_{s \in I(m,v_{j})}\bigcup_{w \in L(j,s)}U^{m+1}_{k}(w).$$

Then
$$U_{k}^{m}(v) = \bigcup_{j=1}^{k}R_{j}.$$

Again by the deep transitivity of the action, $\nu(U^{m+1}_{k}(w))$ is independent of $w$. Using this and induction, it follows that once $\nu(U^{1}_{k}(u))$ is determined for any $u \in \mathbb{N}^{k}$, then $\nu(U^{m}_{k}(v))$ is determined for all $m \ge 1$ and all $v$, and is determined completely then by Dynkin's $\pi$-$\lambda$ Theorem.

Now consider $\mu$. Again since $\mu(\CK_{\le k}(X)) > 0$ and $\mu(\CK_{< k}(X)) = 0$, there exists $\tilde{u} \in \mathbb{N}^{k}$ such that $\mu(U^{1}_{k}(\tilde{u})) > 0$. We have $\nu(U_{k}^{1}(u)) > 0$ for some $u \in \mathbb{N}^{k}$, and deep transitivity implies $\mu(U^{1}_{k}(u)) > 0$ as well. Set $c = \frac{\nu(U^{1}_{k}(u))}{\mu(U^{1}_{k}(u))}$, so $c \cdot \mu(U^{1}_{k}(u)) = \nu(U^{1}_{k}(u))$. The previous paragraph then implies $c \cdot \mu = \nu$. But $\mu$ and $\nu$ are probability measures, so $c = 1$.
\end{proof}

We briefly recall a general fact.
\begin{lemma}
\label{lemma:Fubini}
Let $\mu$ be a nonatomic Borel measure on $X$ and let $k \ge 1$. Let $E = \{x \in X^{k} : x_{i} = x_{j} \textrm{ for some } i \ne 
j, 1 \leq i,j\leq k\}$. Then $\mu^{\otimes k}(E) = 0$ and $\mu^{\CK}_{k}(\CK_{\le j}(X)) = 0$ for every $1 \le j < k$. In particular, $\delta(\mu^{\CK}_{k}) = k$.
\end{lemma}
\begin{proof}
For the first part, the set $E$ is a finite union of (not necessarily disjoint) sets with two coordinates equal, and so by symmetry it suffices to check that $\mu^{\otimes k}(E_{i,j}) = 0$, where 
$E_{i,j} = \{x\in X^k: x_i = x_j\}$ for some fixed $1\leq i,j\leq k$. Then 
$$
\mu^{\otimes k}(E_{i,j}) = 
\int_{X^k} \one_{E_{i,j}}\,d\mu(x_1)\dots d\mu(x_k) = \int_{X^2}
\one_{E_{i,j}}\,d\mu(x_i) d\mu(x_j), 
$$
by using Fubini's Theorem and integrating with respect to all coordinates other than $i$ and $j$.  
Fixing $x_i$, and integrating with respect to $x_j$, this becomes 
$$\int_X \one_{E_{i,j}}\,d\mu(x_j) = \mu(\{x_i\}) 
$$
and so integrating with respect to $x_i$, we have 
$$
\mu^{\otimes k}(E_{i,j}) = 
\int_X \mu(\{x_i\}) \,d\mu(x_i).
$$
As $\mu$ is nonatomic, this integral is $0$. 

For the second part, if $1 \le j < k$ then $\rho_{k}^{-1}(\CK_{\le j}(X)) \subset E$ and so
\begin{equation*}\mu^{\CK}_{k}(\CK_{\le j}(X)) = (\rho_{k})_{*}(\mu^{\otimes k})(\CK_{\le j}(X)) = \mu^{\otimes k}(\rho_{k}^{-1}(\CK_{\le j}(X))) = 0.
\qedhere 
\end{equation*}
That $\delta(\mu^{\CK}_{k}) = k$ then follows. 
\end{proof}

\begin{lemma}
\label{lemma:ergodic}
If $\mu \in \CM_G(X)$, then $\mu^{\CK}_{k} \in \CM_G(\CK_{\le k}(X))$ is ergodic.
\end{lemma}

We note that if $(X,G,\mu)$ is weak mixing (in the sense that $(X^{k},G,\mu)$ is ergodic for $k \ge 2$), then Lemma~\ref{lemma:ergodic} is immediate, since $\mu^{\CK}_{k}$ is a factor of the product measure $\mu^{\otimes k}$ on $X^{k}$. 

\begin{proof}
Since we are assuming $G$ acts deeply transitively on $X$, by Theorem~\ref{thm:deeptransitivityuniqueergodicity} $\mu$ must be nonatomic and unique, and hence ergodic. By Lemma~\ref{lemma:Fubini}, we have $\mu^{\CK}_{k}(\CK_{k}(X)) = 1$ and $\mu^{\CK}_{k}(\CK_{<k}(X)) = 0$, so $\delta(\mu^{\CK}_{k})=k$. Suppose $\mu^{\CK}_{k} = c_{1}\nu_{1} + c_{2}\nu_{2}$ for some $\nu_{1},\nu_{2} \in \mathcal{M}_{G}(\CK_{\le k}(X))$ and $c_{1},c_{2} \ge 0$. Then both $\nu_{1}$ and $\nu_{2}$ must satisfy $\nu_{1}(\CK_{<k}(X)) = \nu_{2}(\CK_{<k}(X)) = 0$, or else one of $c_{1},c_{2}$ must vanish. Thus $\delta(\nu_{1}) = \delta(\nu_{2}) = k$. Since $\mu^{\CK}_{k}(\CK_{k}(X)) = 1$, there must exist an $i$ and $v \in \mathbb{N}^{k}$ such that $\nu_{1}(U^{i}_{k}(v)) > 0$ or $\nu_{2}(U^{i}_{k}(v))> 0$; assume without loss of generality $\nu_{1}(U^{i}_{k}(v)) > 0$. Then $\nu_{1}(U^{1}_{k}(u)) > 0$ for some $u$. Let $e = \mu^{\CK}_{k}(U^{1}_{k}(u)) / \nu_{1}(U^{1}_{k}(u))$, so $e \cdot \nu_{1}(U^{1}_{k}(u)) = \mu^{\CK}_{k}(U^{1}_{k}(u))$. Since $\delta(\nu_{1}) = k$ and $\delta(\mu^{\CK}_{k}) = k$, Lemma~\ref{lemma:knowonevalueknowbasis} then implies that $e \cdot \nu_{1} = \mu^{\CK}_{k}$. As $\nu_{1}$ and $\mu^{\CK}_{k}$ are both probability measures, in fact $e=1$, and $\nu_{1} = \mu^{\CK}_{k}$. Then $\mu^{\CK}_{k} = c_{1}\mu^{\CK}_{k}+c_{2}\nu_{2}$ so $\mu^{\CK}_{k}(1-c_{1}) = c_{2}\nu_{2}$. Since $\mu^{\CK}_{k}$ and $\nu_{2}$ are both probability measures, we have $1-c_{1} = c_{2}$ and hence either $c_{2} = 0$ or $\mu^{\CK}_{k} = \nu_{2}$. It follows altogether that $\mu^{\CK}_{k}$ is ergodic.
\end{proof}

\begin{theorem}
Suppose that $G$ acts deeply transitively on a Cantor set $X$ and suppose $\CM_{G}(X) \ne \emptyset$. Then $|\CM^{e}_{G}(\CK_{\le s}(X))| = s$ for every $s \ge 1$, and if $\nu$ is an ergodic IRC on $\CK_{\le s}(X)$, then $\nu = \mu^{\CK}_{k}$ for some $1 \le k \le s$ where $\mu$ is the unique $G$-invariant Borel probability measure. 
\end{theorem}

\begin{proof}
By Lemmas~\ref{lemma:ergodic} and~\ref{lemma:Fubini},  for every $1 \le k \le s$, the measure $\mu^{\CK}_{k}$ is ergodic and satisfies $\delta(\mu^{\CK}_{k})  = k$. Let $\eta \in \CM_{G}^{e}(\CK_{\le s}(X))$ and suppose $\delta(\eta) = r$. Then $\delta(\eta) = \delta(\mu^{\CK}_{r})$, and so by Lemma~\ref{lemma:knowonevalueknowbasis} it follows that $\eta = \mu^{\CK}_{r}$.  
Altogether it follows that $\mu^{\CK}_{k}$ is the unique ergodic $G$-invariant measure on $\CK_{\le s}(X)$ such that $\delta(\mu^{\CK}_{k})=k$.
\end{proof}

Putting together the results in this section with Theorem~\ref{th:FI}, we have proven Theorem~\ref{th:deeply}.

\begin{theorem}
\label{cor:deeply}
Suppose $G$ acts deeply transitively on a Cantor set $X$. If $\CM_G(X) = \emptyset$, then $(\CK(X),G)$ is uniquely ergodic with unique measure $\delta_{X}$. If $\CM_G(X) \ne \emptyset$ and $\mu \in \CM_G(X)$, then for every IRC $\nu$ of the system there exists a nonnegative sequence $(c_{i})_{i \ge 0}$ satisfying $\sum_{i=0}^{\infty}c_{i}=1$ such that $\nu = \sum_{i=1}^{\infty}c_{i}\mu_{i}^{\CK} + c_{0}\delta_{X}$.
\end{theorem}

\section{Self-joinings of extremely transitive actions}\label{sec:selfjoiningsexttransitivity}
\subsection{Self-joinings and extremely transitive actions}
We study the case that $G$ acts extremely transitively on a Cantor set $X$.  Throughout this section, we assume that $G$ acts extremely transitively with respect to a fixed tree structure $(\mathcal{C}_{i})_{i \in \mathbb{N}} = \{C^{i}_{1},\ldots,C^{i}_{\kappa(i)}\}_{i \in \mathbb{N}}$ on a Cantor set $X$. 

Let $m \ge 1$ and $\mathcal{I} = I_{1},\ldots,I_{r}$ be a partition of $\{1,\ldots,m\}$ where each $I_{j} \ne \emptyset$. There is an embedding
$$\theta_{\mathcal{I}} \colon \prod_{i=1}^{r}X \to \prod_{i=1}^{m}X$$
where $\theta_{\CI}(x_{1},\ldots,x_{r})$ has $x_{k}$ in each coordinate whose index lies in $I_{k}$. Note that $\theta_{\CI}$ is equivariant for the diagonal actions of $G$. We write $\Delta_{\CI}(X)$ for the image of $\theta_{\CI}$. For any partition $\CI$, the subspace $\Delta_{\CI}(X)$ is invariant under the diagonal action of $G$ on $\prod_{i=1}^{m}X$. If $\mu$ is a Borel probability measure on $\prod_{i=1}^{r}X$ invariant under the diagonal action of $G$, then the pushforward $(\theta_{\CI})_{*}(\mu)$ is a Borel probability measure on $\prod_{i=1}^{m}X$ which invariant under the diagonal action of $G$ as well. We write $\mu^{\CI}$ for $(\theta_{\CI})_{*}(\mu)$. It is clear that if $\mu$ is nonatomic, then $\mu^{\CI}$ is nonatomic.

There is a partial ordering on the set $\mathcal{Q}_{m}$ of partitions of $\{1,\ldots,m\}$, defined by $\mathcal{I} \ge \mathcal{J}$ if $\mathcal{I}$ is a refinement of $\mathcal{J}$, i.e. $\mathcal{I}=\{I_{1},\ldots,I_{p}\} \ge \mathcal{J}=\{J_{1},\ldots,J_{q}\}$ if for every $1 \le i \le p$ there exists $1 \le j \le q$ such that $I_{i} \subset J_{j}$. It is readily checked that $\mathcal{J} \le \mathcal{I}$ implies $\Delta_{\CJ}(X) \subset \Delta_{\CI}(X)$. Moreover, the maximal partition $\mathcal{I}_{\max} = \{\{1\},\ldots,\{m\}\}$ is a maximal element in this ordering, and the trivial partition $\mathcal{I}_{\textrm{tr}} = \{1,\ldots,m\}$ is a minimal element. Furthermore, we have $\Delta_{\CI}(X) \cap \Delta_{\CJ}(X) = \Delta_{\CI \lor \CJ}(X)$ where $\CI \lor \CJ$ denotes the join of $\CI$ and $\CJ$.

\begin{lemma}\label{lemma:isosystems}
If $\mu \in \CM_G(X)$ is nonatomic and  $\CI=\{I_{1},\ldots,I_{r}\}$ is a partition of $\{1,\ldots,m\}$, then $(\Delta_{\CI}(X),G,(\mu^{\otimes r})^{\CI})$ is isomorphic to $(\CK_{r}(X),G,\mu^{\CK}_{r})$. Moreover, we have $(\mu^{\otimes r})^{\CI}(\Delta_{\CJ}(X)) = 0$ for all $\CJ \le \CI$ distinct from $\CI$.
\end{lemma}
\begin{proof}
The system $(\Delta_{\CI}(X),G,(\mu^{\otimes r})^{\CI})$ is isomorphic to $(X^{r},G,\mu^{\otimes r})$. Let $E = \{x \in X^{r} : x_{i} = x_{j} \textrm{ for some } i \ne 
j, 1 \leq i,j\leq r\}$. By Lemma~\ref{lemma:Fubini},  $(X^{r},G,\mu^{\otimes r})$ is isomorphic to $(X^{r} \setminus E,G,\mu^{\otimes r})$. Since the map $\rho_{r} \colon X^{r} \to \CK_{\le r}(X)$ takes $X^{r} \setminus E$ bijectively onto $\CK_{r}(X)$, the result then follows since $\mu^{\CK}_{r}$ is the pushforward of $\mu^{\otimes r}$ under $\rho_{r}$.
The last part follows by combining Lemma~\ref{lemma:Fubini}, together with the fact that $\theta_{\CI}^{-1}(\Delta_{\CJ}(X)) \subset E$ and $(\Delta_{\CI}(X),G,(\mu^{\otimes r})^{\CI})$ is isomorphic to $(X^{r},G,\mu^{\otimes r})$.
\end{proof}

\begin{lemma}\label{lemma:freedomofmovement} 
Suppose $(x_{1},\ldots,x_{r}), (y_{1},\ldots,y_{r}) \in X^{r}$ with pairwise distinct entries among  $x_{i}$, and pairwise distinct entries among the  $y_{i}$. Then for every $\varepsilon > 0$, there exists $g \in G$ such that $d(g(x_{i}),y_{i}) < \varepsilon$ for every $1 \le i \le r$.
\end{lemma}
\begin{proof}
Since the $x_{i}$ are distinct, we may choose $m$ and $j_{1},\ldots,j_{r}$ such that $x_{i} \in C^{m}_{j_{i}}$ and the $C^{m}_{j_{i}}$ are pairwise distinct. Likewise, choose $C^{m}_{k_{i}}$ pairwise distinct such that $y_{i} \in C^{m}_{k_{i}}$. 
By Property~\eqref{item:tree-two} in the definition of a tree structure, without loss of generality we may assume $m$ is sufficiently large enough such that the diameters of all $C^{m}_{j}$ are less than $\varepsilon$. 
By the assumption of extreme transitivity, there exists $g \in G$ such that $g(C^{m}_{j_{i}}) = C^{m}_{k_{i}}$ for every $1 \le i \le r$. 
It follows that for all such $i$, we have $g(x_{i}) \in g(C^{m}_{j_{i}}) = C^{m}_{k_{i}}$, and since $\diam(C^{m}_{k_{i}}) < \varepsilon$, the statement follows. 
\end{proof}

Given $x \in X^{m}$, for $1 \le i,j \le m$ we write $i \sim_{x} j$ if $x_{i} = x_{j}$. We call the partition of $\{1,\ldots,m\}$ defined by the relation $\sim_{x}$ to be the {\em $x$-partition $\mathcal{I}_{x}$}. Note that $x \in \Delta_{\CI_{x}}(X)$.

\begin{theorem}\label{thm:topologicalselfjoinings}
If $Y \subset X^{m}$ is a nonempty compact $G$-invariant set, then $Y = \bigcup_{\CI_{j}}\Delta_{\mathcal{I}_{j}}(X)$ for some partitions $\mathcal{I}_{j}$ of $\{1,\ldots,m\}$. Moreover, if $\Delta_{\CI}(X) \subset Y$ for some partition $\CI$ of $\{1,\ldots,m\}$, then $\Delta_{\mathcal{J}}(X) \subset Y$ for every $\mathcal{J} \le \mathcal{I}$.
\end{theorem}
\begin{proof}
Define $\mathcal{J} = \{\mathcal{I}_{y} : y \in Y\}$. Since $y \in \Delta_{\CI_{y}}(X)$ and $\Delta_{\CI}(X)$ is $G$-invariant for every partition $\CI$, it follows that $Y \subset \bigcup_{y \in Y}\Delta_{\CI_{y}}(X)$. We show that we also have $\bigcup_{y \in Y}\Delta_{\CI_{y}}(X) \subset Y$.

Let $y \in Y$ and $x \in \Delta_{\mathcal{I}_{y}}(X)$. Let $(x_{1},\ldots,x_{r}) = \theta_{\CI_{y}}^{-1}(x), (y_{1},\ldots,y_{r}\} = \theta_{\CI_{y}}^{-1}(y)$. 
Let $\varepsilon > 0$. Since $X$ is Cantor, there exists $(x^{\prime}_{1},\ldots,x^{\prime}_{r})$ such that $d(x_{i},x^{\prime}_{i})< \varepsilon$ for each $1 \le i \le r$ and the $x^{\prime}_{i}$ are pairwise distinct. 
By assumption, the $y_{i}$ are pairwise distinct for $1 \le i \le r$. Then by Lemma~\ref{lemma:freedomofmovement}, there exists $g \in G$ such that $d(g(y_{i}),x_{i}^{\prime}) < \varepsilon$ for every $1 \le i \le r$, and hence $d(g(y_{i}),x_{i}) < 2 \varepsilon$ for such $i$. 
It follows that $(x_{1},\ldots,x_{r})$ lies in the closure of the $G$-orbit of $(y_{1},\ldots,y_{r})$, and hence $\Delta_{\CI_{y}}(X) \subset Y$ since $y \in Y$ and $Y$ is compact and $G$-invariant. 
Since $y \in Y$ is arbitrary, we have that  $\bigcup_{y \in Y}\Delta_{\CI_{y}}(X) \subset Y$.

The last part is immediate, since $\Delta_{\CJ}(X) \subset \Delta_{\CI}(X)$ for every $\CJ \le \CI$.
\end{proof}

Suppose $\CI = \{I_{1},\ldots,I_{r}\}$ is a partition of $\{1,\ldots,m\}$. Given $i \ge 1$ and $u \in \mathbb{N}^{r}$, we define
$$U^{i}_{\CI}(u) = \{(x_{1},\ldots,x_{m}) \in \Delta_{\CI}(X) : x_{j} \in C^{i}_{u_{j}} \textrm{ if } j \in I_{j}\}.$$
Note that $U^{i}_{\CI}(u)$ is the image of $C^{i}_{u_{1}} \times \cdots \times C^{i}_{u_{r}}$ under the embedding $\theta_{\CI} \colon X^{r} \to X^{m}$. Since sets of the form $C^{i}_{u_{1}} \times \cdots \times C^{i}_{u_{r}}$ form a basis for $X^{r}$ and $\theta_{\CI}$ is a homeomorphism onto its image, it follows that the collection of $U^{i}_{\CI}(u)$ form a basis for $\Delta_{\CI}(X)$ and also a $\pi$-system. Moreover, the collection of $U^{i}_{\CI}(u)$ where the entries of $u$ are distinct form a basis for $\Delta_{\CI}(X) \setminus \bigcup_{\CJ \le \CI} \Delta_{\CJ}(X)$.

\begin{lemma}\label{lemma:knowoneknowbasisextreme}
Let $m \ge 1$ and suppose $\CI = \{I_{1},\ldots,I_{k}\}$ is a partition of $\{1,\ldots,m\}$. Suppose $\mu,\nu$ are $G$-invariant probability measures on $\Delta_{\CI}(X)$ such that $\mu(\Delta_{J}(X)) = \nu(\Delta_{\CJ}(X)) = 0$ for all $\CJ \le \CI$. Then $\mu = \nu$.
\end{lemma}
In other words, if a $G$-invariant Borel probability measure on $\Delta_{\mathcal{I}}(X)$ exists satisfying $\mu(\Delta_{\mathcal{J}}(X)) = 0$ for all $\mathcal{J} \le \mathcal{I}$, then it is unique. 

\begin{proof}
The proof is similar to that of Lemma~\ref{lemma:knowonevalueknowbasis}. We start by considering the measure $\nu$. By assumption we have that $\nu(\Delta_{\CI}(X)) > 0$ and $\nu(\Delta_{\CJ}(X)) = 0$ for all $\CJ \le \CI$, and so there exists $u \in \mathbb{N}^{k}$ such that $\nu(U^{1}_{\CI}(u)) > 0$. Suppose $U^{1}_{\CI}(u) = \theta_{\CI}(C^{1}_{j_{1}} \times \cdots \times C^{1}_{j_{k}})$ and $U^{1}_{\CI}(u^{\prime}) = \theta_{\CI}(D^{1}_{j_{1}} \times \cdots \times D^{1}_{j_{k}})$ is another such set for some $u^{\prime} \in \mathbb{N}^{k}$. By the extreme transitivity of the action, there exists $g \in G$ such that $(g(C^{i}_{j_{1}}),\ldots,g(C^{i}_{j_{k}})) = (D^{i}_{j_{1}},\ldots,D^{i}_{j_{k}})$ and hence $g(U^{1}_{\CI}(u)) = U^{1}_{\CI}(u^{\prime})$. Since $\nu$ is $G$-invariant, it follows that
$$\nu(U^{1}_{\CI}(u)) = \nu(U^{1}_{\CI}(u^{\prime})).$$

Let $m \ge 1$ and consider some $U^{m}_{\CI}(v) = \theta_{\CI}(C^{m}_{v_{1}} \times \cdots \times C^{m}_{v_{k}})$ for some $v \in \mathbb{N}^{k}$. 
For each $1 \le \ell \le k$, using the tree structure can write each $C^{m}_{v_{l}}$ as a disjoint union 
$$C^{m}_{v_{\ell}} = \bigcup_{r \in I(m,v_{l})}C^{m+1}_{r}.$$

Then we can write $U^{m}_{\CI}(v)$ as a disjoint union of sets of the form $U^{m+1}_{\CI}(w)$ for various $w \in \mathbb{N}^{k}$ as follows. 
Let $J = \prod_{i=1}^{k}I(m,v_{i})$. For $1 \le j \le k$ and $s \in I(m,v_{j})$, 
set $L(j,s) = \{w \in J : w_{j} = s\}$ and define
$$R_{j} = \bigcup_{s \in I(m,v_{j})}\bigcup_{w \in L(j,s)}U^{m+1}_{\CI}(w).$$  
Then
$$U_{\mathcal{I}}^{m}(v) = \bigcup_{j=1}^{k}R_{j}.$$

By the extreme transitivity of the action, $\nu(U^{m+1}_{\CI}(w))$ is independent of $w$. Using this and induction, it follows that once $\nu(U^{1}_{\CI}(u))$ is determined for any $u \in \mathbb{N}^{k}$, then $\nu(U^{m}_{v})$ is determined for all $m \ge 1$ and all $v\in\N^k$, and then determined completely by Dynkin's $\pi$-$\lambda$ Theorem.

Next consider the measure $\mu$. Since $\mu(\Delta_{\CI}(X)) > 0$ and $\mu(\Delta_{\CJ}(X)) = 0$ for all $\CJ \le \CI$ distinct from $\CI$, there exists $w \in \mathbb{N}^{k}$ such that $\mu(U^{1}_{\CI}(w)) > 0$.
Let $c = \frac{\nu(U^{1}_{\CI}(w))}{\mu(U^{1}_{\CI}(w))}$, so $c \cdot \mu(U^{1}_{\CI}(w)) = \nu(U^{1}_{\CI}(w))$. It follows from the calculations that $c \cdot \mu = \nu$. Since $\mu$ and $\nu$ are probability measures, we have that $c=1$ and so $\mu = \nu$. 
\end{proof}

\begin{lemma}\label{lemma:muIergodic}
Let $\CI$ be a partition of $\{1,\ldots,m\}$ where $|\mathcal{I}|=r$, and let $\mu \in \CM_G(X)$. Then $(\mu^{\otimes r})^{\CI} \in \CM_G(\Delta_{\CI}(X))$ is ergodic.
\end{lemma}
\begin{proof}
Since $G$ acts extremely transitively, it acts deeply transitively as well, and so $\mu^{\CK}_r$ is ergodic by Lemma~\ref{lemma:ergodic}. But $(\Delta_{\CI}(X),G,(\mu^{\otimes r})^{\CI})$ is isomorphic to $(\CK_{r}(X),G,\mu^{\CK}_{r})$ by Lemma~\ref{lemma:isosystems}, and so $(\mu^{\otimes r})^{\CI}$ is also ergodic.
\end{proof}

\begin{theorem}
Let $m\geq 1$.  If $\CM_G(X) \ne \emptyset$ and $\nu$ is an ergodic measure on $X^{m}$, then $\nu = (\mu^{\otimes r})^{\CI}$ for some partition $\CI$ of $\{1,\ldots,m\}$ of size $r$, where $\mu$ is the unique $G$-invariant Borel probability measure on $X$. Hence if $\CM_G(X) \ne \emptyset$, then $|\CM_{G}^{e}(X^{m})| = B_{m}$ for every $m\geq 1$, where $B_{m}$ denotes the $m\textsuperscript{th}$ Bell number.
\end{theorem}
\begin{proof}
Let $\mu \in \mathcal{M}_{G}(X)$. Since $X$ is infinite and $G$ acts extremely transitively, Proposition~\ref{prop:deeptransitivityprime} implies that the system $(X,G)$ is minimal and so it follows that the measure $\mu$ is nonatomic. Then for any partition $\CI$ of $\{1,\ldots,m\}$ of size $r$, Lemma~\ref{lemma:isosystems} shows that $(\mu^{\otimes r})^{\CI}$ is nonatomic and satisfies $(\mu^{\otimes r})^{\CI}(\Delta_{\CJ}(X)) = 0$ for all $\CJ \le \CI$ distinct from $\CI$, and is ergodic by Lemma~\ref{lemma:muIergodic}. Let $\eta \in \CM_{G}^{e}(X^{m})$. By Theorem~\ref{thm:topologicalselfjoinings}, the support of $\eta$ is $\bigcup_{\CI_{j}}\Delta_{\CI_{j}}(X)$ for some partitions $\CI_{j}$ of $\{1,\ldots,m\}$. Thus there exists some partition $\CI$ such that $\eta(\Delta_{\CI}(X)) > 0$. Since $\Delta_{\CI}(X)$ is $G$-invariant and $\eta$ is ergodic, we have $\eta(\Delta_{\CI}(X)) = 1$. Without loss of generality we may assume that $\eta(\Delta_{\CJ}(X)) = 0$ for all $\CJ \le \CI$, since by ergodicity if $\eta(\Delta_{\CJ}(X)) > 0$ then $\eta(\Delta_{\CJ}(X)) = 1$.  
Then Lemma~\ref{lemma:knowoneknowbasisextreme} implies $\eta = (\mu^{\otimes |\mathcal{I}|})^{\CI}$. It follows that $|\CM_{G}^{e}(X^{m})|$ is equal to the number of partitions of $\{1,\ldots,m\}$ which is given by $m\textsuperscript{th}$ Bell number $B_{m}$.
\end{proof}

\subsection{Extremely transitive actions and $\infty$-fold topological minimal self-joinings}
For a group $G$, let $Z(G)$ denote its center. 
Following~\cite{BDP}, we say an action of $G$ on $X$ has \emph{$n$-fold topological minimal self-joinings (TMSJ)} if for every $n$ points $x_{1},\ldots,x_{n}$ in $X$ with no pair of them lying in the same $Z(G)$-orbit, the orbit of $(x_{1},\ldots,x_{n})$ in $X^{n}$ is dense under the diagonal action of $G$. The action has \emph{$\infty$-fold topological minimal self-joinings} if it has $n$-fold TMSJ for every $n \ge 1$. It is straightforward to check that if $G$ acts extremely transitively on $X$, then it has $\infty$-fold TMSJ. We summarize several consequences proved here and using~\cite{BDP}.

\begin{proposition}
If $G$ acts extremely transitively on $X$, then it has $\infty$-fold topological minimal self-joinings, and all of the following hold:
\begin{enumerate}
\item\label{item:trivial-center}
The center of $G$ is trivial; 
\item\label{item:prime}
The system $(X,G)$ is prime; 
\item
If $G$ is amenable, the $h_{\textrm{top}}(X,G) = 0$; 
\item
The centralizer of $G$ in $\textrm{Homeo}(X)$ is trivial; 
\item
The action is expansive.
\end{enumerate}
\end{proposition}
\begin{proof}
Property~\ref{item:trivial-center} is proven in Proposition~\ref{prop:deeptransitivecenter} and Property~\ref{item:prime} in Proposition~\ref{prop:deeptransitivityprime}, since extreme transitivity implies deep transitivity. The remaining statements are all proved in~\cite[Section 4]{BDP}. 
\end{proof}

\subsection{Prefix-permutations of real numbers}
Let $\mathcal{S}_{n}$ denote the locally-finite group defined in~\ref{subsubsection:affullgroups}. Recall that the action of $\mathcal{S}_{n}$ on $X_{n}^{+} = \{0,\ldots,n-1\}^{\mathbb{N}}$ is 
defined for a given $\tau \in \sym(n^{k})$ and  point $x = x_{1}x_{2}\ldots \in X_{n}^{+}$ by setting  $\tau(x) = \tau(x_{1}\ldots x_{k})x_{k+1}\ldots$ and this  action of $\mathcal{S}_{n}$ is extremely transitive.

The group $\CS_{n}$ also acts on $\mathbb{T} = \mathbb{R}/\mathbb{Z}$ by prefix permutation of numbers in base $n$, in the following sense. For each $n \ge 2$, we have the base $n$ coding map $\pi_{n} \colon X_{n}^{+} \to \mathbb{T}$ defined by $\pi_{n}(x) = \sum_{k=1}^{\infty}\frac{x_{k}}{n^{k}}$. Given $k \ge 1$ and $w \in \{0,\ldots,n-1\}^{k}$, define $e(w) = \pi_{n}(w000\ldots) \in \mathbb{T}$ and set $I_{w} = [\frac{e(w)}{n^{k}},\frac{e(w)+1}{n^{k}})$. For each $k$ we have $\sym(n^{k})$ act on $\{I_{w} : w \in \{0,\ldots,n-1\}^{k}$ by $\tau(I_{w}) = I_{\tau(w)}$, and this defines a measure-preserving action of $\mathcal{S}_{n}$ on the measure space $(\mathbb{T},\lambda)$ where $\lambda$ is Lebesgue measure. Note that this is not a topological action, as the maps $g\in\sym(n^k)$ do not act continuously.

For a system $(X,G)$, given a F\o{}lner sequence $F_{k}$ for $G$, a compact set $Y \subset X$, and $\varepsilon > 0$, define
$$\mathcal{Z}_{k}(Y,\varepsilon) = \frac{1}{|F_{k}|}\big\vert\{g \in F_{k} : d_{H}(g(Y),X) < \varepsilon\}\big\vert.$$
We  prove the following theorem.

\begin{theorem}\label{thm:prefixpermutcircle}
Let $n \ge 2,$ let $\mathcal{S}_{n}$ act on $\mathbb{T}$ by base $n$ prefix-permutation and let $Y \subset \mathbb{T}$ be an infinite set. Then for every $\varepsilon > 0$, we have $\mathcal{Z}_{k}(Y,\varepsilon) \to 1$ as $k \to \infty$ where $F_{k} = \sym(n^{k})$. 
\end{theorem}

To prove Theorem~\ref{thm:prefixpermutcircle}, we first prove the analogous result for the action of $\CS_{n}$ on the one-sided full shift $X_{n}^{+}$. We begin with a lemma.
\begin{lemma}\label{lemma:permutationdensities}
Let $n \ge 2$. If $r \ge 1$ is an integer, $Y\subset X_{n}^{+}$ is an infinite compact subset, and $\varepsilon < \frac{1}{2r}$, then $$\frac{1}{n^{m}!}\big|\{g \in \sym(n^{m}) : d(g(Y),\CK_{\le r}(X_{n}^{+})) < \varepsilon\}\big| \to 0 \quad  \text{ as } m \to \infty.$$
\end{lemma}
\begin{proof}
Let $Y \subset X_{n}^{+}$ be an infinite compact set and $0 < \varepsilon < \frac{1}{2r}$. Without loss of generality, we assume $\varepsilon = \frac{1}{2^{k}} < \frac{1}{2r}$ for some $k \ge 1$. Let $\CC_{m} = \{C^{m}_{1},\ldots,C^{m}_{n^{m}}\}$ be the tree structure on $X_{n}^{+}$ whose level $m$ partition is given by the cylinder sets of words of length $m$, and we order the $C_{i}^{m}$ by lexicographical ordering of the set of words of length $m$. Then any ball of radius $\varepsilon$ in $X_{n}^{+}$ intersects $n^{m-k}$ many of the $C^{m}_{i}$. Thus if $Z$ is any compact set which satisfies $d(Z,\CK_{\le r}(X_{n}^{+})) < \varepsilon$, then $Z \subset \bigcup_{j=1}^{r}(x_{j}-\varepsilon,x_{j}+\varepsilon)$ for some $x_{1},\ldots,x_{r} \in X_{n}^{+}$, and hence $S_{m}(Z) = \{j : Z \cap C^{m}_{j} \ne \emptyset\}$ is contained in a union of at most $r$ many sets of indices of the form $[j,j+n^{m-k})$. We call such a union an $r$-block.

Let $\alpha(m) = |S_{m}(Y)|$. Note that $S_{m}(g(Y)) = g(s_{m}(Y))$ for $g \in \sym(n^{m})$. Set
$$B(m) = |\{g \in \sym(n^{m}) : g(S_{m}(Y)) \textrm{ is contained in an } r\textrm{-block}\}|.$$
We want to show that $\frac{B(m)}{n^{m}!} \to 0$ as $m \to \infty$. This is equivalent to showing that the probability $p(m)$ that a uniformly random subset $S$ of $\{1,\ldots,n^{m}\}$ is contained in an $r$-block tends  to zero as $m \to \infty$.

Set $\ell(m) = n^{m-k}$ and $A_{m} = \{1,\ldots,n^{m}\}$. If $S$ is contained in an $r$-block, then there are $r$ points $a_{1},\ldots,a_{r} \in A_{m}$ such that $S \subset \bigcup_{i=1}^{r}[a_{i},a_{i}+\ell(m))$. So to build sets $S \subset A_{m}$ which are contained in an $r$-block, we choose $r$ elements of $S$ to play the role of the $a_{i}$s, choose the values for $a_{i}$ in $A_{m}$, and then choose the remaining $\alpha(m)-r$ points for $S$ lying inside the union of the $[a_{i},a_{i}+\ell(m))$. There are at most $\binom{\alpha(m)}{r}$ choices for the first step, then $(n^{m})^{r}$ choices for the second step, then at most $\binom{r\ell(m)}{\alpha(m)-r}$ choices for the last step, so a total of $\binom{\alpha(m)}{r}(n^{m})^{r}\binom{r\ell(m)}{\alpha(m)-r}$ choices for such sets $S$. Thus the probability is
\begin{equation}\label{eqn:probbound1}
p(m) = \frac{\binom{\alpha(m)}{r}(n^{m})^{r}\binom{r\ell(m)}{\alpha(m)-r}}{\binom{n^{m}}{\alpha(m)}}.
\end{equation}

Write
$$\frac{\binom{r\ell(m)}{\alpha(m)-r}}{\binom{n^{m}}{\alpha(m)}} = \frac{\binom{r\ell(m)}{\alpha(m)-r}}{\binom{n^{m}}{\alpha(m)-r}} \frac{\binom{n^{m}}{\alpha(m)-r}}{\binom{n^{m}}{\alpha(m)}}.$$
 Set $\gamma = \frac{r}{n^{k}}$. Note since $\frac{r}{2^{k}} < 1$ by assumption, we have $\gamma < 1$. A straightforward calculation shows that
$$\frac{\binom{r\ell(m)}{\alpha(m)-r}}{\binom{n^{m}}{\alpha(m)-r}} \le \left( \frac{r\ell(m)}{n^{m}} \right)^{\alpha(m)-r} = \gamma^{\alpha(m)-r}.$$
For the second term, first note we must have $\alpha(m)-r \le r\ell(m)$ for $S$ to be contained in an $r$-block, so $\alpha(m) \le \frac{r}{n^{k}}n^{m}$. Then another straightforward calculation shows that the second term satisfies
$$\frac{\binom{n^{m}}{\alpha(m)-r}}{\binom{n^{m}}{\alpha(m)}} \le \left( \frac{\alpha(m)}{cn^{m}} \right)^{r}$$
where $c = c(r,k)$ is some constant only depending on $r$ and $k$. Plugging these into~\eqref{eqn:probbound1} we get
$$p(m) \le d \binom{\alpha(m)}{r} \gamma^{\alpha(m)-r} \alpha(m)^{r}$$
where $d = d(r,k)$ is some constant depending only on $r$ and $k$. Lastly, since $\binom{\alpha(m)}{r} \le \alpha(m)^{r}$, we get
$$p(m) \le d \alpha(m)^{2r}\gamma^{\alpha(m)-r}.$$
Now since $Y$ is infinite, we have $\alpha(m) \to \infty$ as $m \to \infty$, so $p(m) \to 0$ as $m \to \infty$.
\end{proof}

\begin{theorem}\label{thm:prepermutesymbolic}
Let $n \ge 2,$ let $\mathcal{S}_{n}$ act on $X_{n}^{+}$ by prefix-permutation and let $Y \subset X_{n}^{+}$ be an infinite set. Then for every $\varepsilon > 0$, we have $\mathcal{Z}_{k}(Y,\varepsilon) \to 1$ as $k \to \infty$ where $F_{k} = \sym(n^{k})$. 
\end{theorem}

\begin{proof}
Let $Y \subset X_{n}^{+}$ be infinite. It suffices to prove the result for the closure of $Y$ in $X_{n}^{+}$, and so without loss of generality we assume $Y$ is compact.

Consider the sequence of IRCs
$$\mu_{k} = \frac{1}{|F_{k}|}\sum_{g \in F_{k}}\delta_{gY},$$
where $F_{k} = \sym(n^{k})$. It suffices to show that $\mu_{k} \to \delta_{X_{n}^{+}}$ as $k\to\infty$. Suppose $\mu_{k_{j}}$ is a subsequence converging to some $\mu$. It follows from Lemma~\ref{lemma:permutationdensities} $\mu(\CK_{\le r}(X_{n}^{+})) = 0$ for every $r \ge 1$, and thus $\mu$ is a nonfinitary IRC. Since $\CS_{n}$ acts extremely transitively on $X_{n}^{+}$, Theorem~\ref{th:FI} implies the system is IC-rigid, and hence $\mu = \delta_{X_{n}^{+}}$. It follows that the only accumulation point of the sequence $\mu_{k}$ is $\delta_{X_{n}^{+}}$, and hence $\mu_{k} \to \delta_{X_{n}^{+}}$ as $k \to \infty$.  
\end{proof}

We have now assembled the tools to prove Theorem~\ref{thm:prefixpermutcircle}.
\begin{proof}
We note that the map $\pi_{n} \colon X_{n}^{+} \to \mathbb{T}$ is not equivariant for the respective actions of $\CS_{n}$.  
The map $\pi_{n}$ has fibers of size one or two, and fibers of size two have the form $\pi^{-1}(x) = \{x^{\ell},x^{u}\}$ where $x^{\ell} = x_{1}\ldots x_{k} a 00\ldots$, $x^{u} = x_{1} \ldots x_{k} (a-1) (n-1)(n-1)\ldots$ and $a \in \{1,\ldots,n-1\}$.  
Define $\xi \colon \mathbb{T} \to X_{n}^{+}$ by $\xi(x) = \pi^{-1}(x)$ if $|\pi^{-1}(x)| = 1$, and $\xi(x) = x^{\ell}$ if $|\pi^{-1}(x)|=2$. Then $x \in I_{w}$ if and only if $\xi(x) \in [w]$, and $\xi gx  = g\xi x$ for every $g \in \CS_{n}$ and $x \in \mathbb{T}$. Let $Z = \xi(Y)$ and let $\varepsilon > 0$. Then $Z$ is an infinite subset of $X_{n}^{+}$, and so by Theorem~\ref{thm:prepermutesymbolic} we have $\mathcal{Z}_{k}(Z,\varepsilon) \to 1$ as $k \to \infty$. Since $\xi(g(Y)) = g\xi(Y)$ for every $g \in G$, if $gZ$ is $\frac{1}{2^{p}}$-dense in $X_{n}^{+}$, then $\pi_{n}(gZ)$ is $\frac{1}{n^{p}}$-dense in $\mathbb{T}$ and rhe result follows. 
\end{proof}

\end{document}